\documentclass [11pt,twoside,a4paper]{article}

\usepackage{amsfonts}
\usepackage{amsthm}
\usepackage{amsmath}
\usepackage{amstext}
\usepackage{amssymb}
\usepackage{mathrsfs}
\usepackage{amscd}
\usepackage{xypic}
\usepackage{epsf} 
\usepackage{graphicx} 
\usepackage{fancybox} 
\usepackage{color} 
\usepackage{fancyhdr}
\usepackage[hang,footnotesize]{caption2} 

\usepackage{amssymb}
\usepackage{titletoc}
\usepackage{mathrsfs}
\usepackage{latexsym,amssymb}

\usepackage{indentfirst}
\usepackage{txfonts}

\allowdisplaybreaks

\usepackage[
top=1.5cm,
headheight=1cm,
body={132mm,205mm},
centering
]{geometry}
\def\mathcal{\mathscr}
\newfont{\aaa}{cmb10 at 19pt}
\newfont{\bbb}{cmb10 at 11pt}

\newtheorem{thm}{Theorem}[section]
\newtheorem{lem}[thm]{Lemma}
\newtheorem{pro}[thm]{Proposition}
\newtheorem{cor}[thm]{Corollary}

\theoremstyle{definition}
\newtheorem{defn}[thm]{Definition}
\theoremstyle{remark}
\newtheorem{rem}[thm]{Remark}

\numberwithin{equation}{section}

\def\bint{{\ifinner\rlap{\bf\kern.30em--}
\int\else\rlap{\bf\kern.35em--}\int\fi}\ignorespaces}

\def\sbint{{\ifinner\rlap{\bf\kern.32em--}
\hspace{0.078cm}\int\else\rlap{\bf\kern.45em--}\int\fi}\ignorespaces}

\def\rr{{\mathbb R}}
\def\rn{{\mathbb{R}^n}}
\def\cc{{\mathbb C}}

\def\nn{{\mathbb N}}
\def\zz{{\mathbb Z}}

\def\fz{\infty }
\def\az{\alpha}
\def\bz{\beta}
\def\dz{\delta}
\def\ez{\epsilon}

\def\lz{\lambda}

\def\lf{\left}
\def\r{\right}
\def\ls{\lesssim}
\def\noz{\nonumber}

\def\loc{{\mathrm{loc}}}
\DeclareMathOperator{\supp}{supp}

\def\lz{{\lambda}}

\makeatletter
\def\@evenhead{
\vbox{\hbox to \textwidth {}{\hspace{0mm}{\footnotesize
\thepage}}{\hspace{3cm} {\footnotesize
{Hongchao JIA, Jin TAO, Dachun YANG, Wen YUAN, Yangyang ZHANG}}}
\protect\vspace{1truemm}\relax \hrule depth0pt height0.15truemm
width\textwidth}}
\def\@evenfoot{}
\def\@oddhead{\vbox{\hbox to \textwidth
{{\hspace{0cm}{\footnotesize Fractional integrals
on congruent JNC spaces}\hfill{\footnotesize
\thepage}}\hspace{0mm}}{} \protect\vspace{1truemm}\relax\hrule
depth0pt height0.15truemm width\textwidth}}
\def\@oddfoot{}
\makeatother

\begin{document}

\thispagestyle{empty}

\fancypagestyle{firststyle}
{
\renewcommand{\topmargin}{-9mm}
\fancyhead[lO,RE]{\footnotesize Front. Math. China \\
https:/\!/doi.org/10.1007/s11464-021-XXXX-X\\[3mm]
}
\fancyhead[RO,LE]{\scriptsize \bf 
} \fancyfoot[CE,CO]{}}
\renewcommand{\headrulewidth}{0pt}


\setcounter{page}{1}
\qquad\\[5mm]

\thispagestyle{firststyle}

\noindent{\aaa{Boundedness of fractional integrals on\\[2mm]
special John--Nirenberg--Campanato and\\[2mm]
Hardy-type spaces via congruent cubes}}\\[1mm]

\noindent{\bbb Hongchao JIA,\quad Jin TAO,\quad Dachun YANG,\quad
\\ Wen YUAN,\quad Yangyang ZHANG}\\[-1mm]

\noindent\footnotesize{Laboratory of Mathematics and Complex Systems
(Ministry of Education of China),
School of Mathematical Sciences, Beijing Normal University,
Beijing 100875, People's Republic of China}\\[6mm]

\vskip-2mm \noindent{\footnotesize$\copyright$\,Higher Education
Press 2021} \vskip 4mm

\normalsize\noindent{\bbb Abstract}\quad
Let $p\in[1,\infty]$, $q\in[1,\infty)$,
$s\in\mathbb{Z}_+:=\mathbb{N}\cup\{0\}$,
$\alpha\in\mathbb{R}$, and $\beta\in (0,1)$.
In this article, the authors first find a reasonable
version $\widetilde{I}_{\beta}$ of the
(generalized) fractional integral $I_{\beta}$
on the special John--Nirenberg--Campanato space via congruent cubes,
$JN_{(p,q,s)_\alpha}^{\mathrm{con}}(\mathbb{R}^n)$,
which coincides with the Campanato space
$\mathcal{C}_{\alpha,q,s}(\mathbb{R}^n)$ when $p=\infty$.
To this end, the authors introduce
the vanishing moments up to order $s$ of $I_{\beta}$.
Then the authors prove that
$\widetilde{I}_{\beta}$ is bounded
from $JN_{(p,q,s)_\alpha}^{\mathrm{con}}(\mathbb{R}^n)$
to $JN_{(p,q,s)_{\alpha+\beta/n}}^{\mathrm{con}}(\mathbb{R}^n)$
if and only if $I_{\beta}$ has the vanishing moments up to order $s$.
The obtained result is new even when $p=\infty$ and $s\in\mathbb{N}$.
Moreover, the authors show that $I_{\beta}$
can be extended to a unique continuous
linear operator from the Hardy-kind space
$HK_{(p,q,s)_{\alpha+\beta/n}}^{\mathrm{con}}(\mathbb{R}^n)$,
the predual of $JN_{(p',q',s)_{\alpha+\beta/n}}^{\mathrm{con}}(\mathbb{R}^n)$
with $\frac{1}{p}+\frac{1}{p'}=1=\frac{1}{q}+\frac{1}{q'}$, to
$HK_{(p,q,s)_{\alpha}}^{\mathrm{con}}(\mathbb{R}^n)$
if and only if $I_{\beta}$ has the vanishing moments up to order $s$.
The proof of the latter boundedness strongly depends on the dual relation
$(HK_{(p,q,s)_{\alpha}}^{\mathrm{con}}(\mathbb{R}^n))^*
=JN_{(p',q',s)_\alpha}^{\mathrm{con}}(\mathbb{R}^n)$,
the properties of molecules of
$HK_{(p,q,s)_\alpha}^{\mathrm{con}}(\mathbb{R}^n)$,
and a crucial criterion for the boundedness
of linear operators on
$HK_{(p,q,s)_\alpha}^{\mathrm{con}}(\mathbb{R}^n)$.\vspace{0.3cm}

\footnotetext{Received September 20, 2021; accepted December 24, 2021\\
\hspace*{5.8mm}Corresponding author: Dachun YANG, E-mail:
dcyang@bnu.edu.cn}

\noindent{\bbb Keywords}\quad John--Nirenberg space,
Campanato space, Hardy space, fractional integrals,
molecule.\\[1mm]
{\bbb MSC2020}\quad 42B20, 47A30, 42B30, 46E35, 42B25, 42B35\\[0.4cm]

\section{Introduction\label{s1}}

In this article, for any given measurable set $E\subset \rn$ and
for any given $q\in(0,\infty]$,
the \emph{Lebesgue space} $L^q(E)$
is defined to be the set of all measurable
functions $f$ on $E$ such that
$$\|f\|_{L^q(E)}:=
\begin{cases}
\displaystyle
\lf[\int_{E}|f(x)|^q\, dx\r]^{\frac{1}{q}}
&\text{if}\quad q\in(0,\fz),\\
\displaystyle
\mathop{\mathrm{ess\,sup}}_{x\in E}|f(x)| &\text{if}\quad q=\fz
\end{cases}$$
is finite. The \textit{space} $L^q_{\mathrm{loc}}(\rn)$
is defined to be the set of all measurable functions $f$ on $\rn$ such that,
for any given bounded measurable set $F\subset \rn$,
$f\mathbf{1}_F\in L^q(\rn)$,
here and thereafter, we use $\mathbf{1}_F$ to
denote the \emph{characteristic function} of $F$.

Recall that the fractional integral
$\mathcal{I}_{\beta}$, with $\beta\in(0,n)$, is defined by setting,
for any $f\in L^q(\rn)$ with $q\in [1,\frac{n}{\beta})$,
and almost every $x\in \rn$,
\begin{align}\label{cla-I}
\mathcal{I}_{\beta}(f)(x)
:=\int_{\rn}\frac{f(y)}{|x-y|^{n-\beta}}\,dy.
\end{align}
The well-known Hardy--Littlewood--Sobolev theorem
shows that $\mathcal{I}_{\beta}$
is bounded from $L^q(\rn)$ to $L^{\widetilde{q}}(\rn)$
with $q\in(1,\frac{n}{\beta})$ and
$\frac{1}{\widetilde{q}}:=\frac{1}{q}-\frac{\beta}{n}$,
which was first established by Hardy and Littlewood \cite{hl1928}
and Sobolev \cite{s1938}
(see also \cite{EMS1970}), and plays a key role in potential
theory and partial differential equations;
see, for instance, \cite{mv1995, Po1997, R1996}.
From then on, many studies on fractional integrals
have been done.
For instance, Sawyer and Wheeden \cite{sw1992} studied the
weighted inequalities for fractional integrals
on Euclidean and homogeneous spaces. Nakai \cite{N2001}
introduced the generalized fractional integral
and extended the Hardy--Littlewood--Sobolev theorem to the Orlicz space.
Later, Sawano and Shimomura \cite{ss2017}
studied the boundedness of the generalized
fractional integral on generalized Morrey spaces.
Moreover,  Nakai \cite{N10, N17} studied the
boundedness of fractional integrals
on Campanato spaces with variable growth conditions,
and on their predual spaces, namely, some Hardy-type spaces.
Very recently, Ho \cite{Ho2021} studied the Erd\'elyi--Kober
fractional integral operator on ball Banach function spaces.
Chen and Sun \cite{cs2021} studied the
extension of multilinear fractional integrals to
linear operators on mixed-norm Lebesgue spaces.
We refer the reader to \cite{dll2003, hj2015, jzwh2017, Ho2020} for more
studies on fractional integrals on Hardy-type spaces,
and to \cite{an2019,lx2020, mx2017, NS2012, sl2014} on Campanato-type spaces.

In this article, we extend the Hardy--Littlewood--Sobolev theorem
to the following special John--Nirenberg--Campanato space via congruent cubes,
$JN_{(p,q,s)_\alpha}^{\mathrm{con}}(\rn)$, which was
introduced in \cite{jtyyz1} as a John--Nirenberg-type space.
First, recall that John and Nirenberg \cite{JN}
introduced the well-known space $\mathrm{BMO}\,(Q_0)$ and,
in the same article \cite{JN},
as a generalized version of the space $\mathrm{BMO}\,(Q_0)$,
they also studied the following John--Nirenberg space $JN_p(Q_0)$.
In what follows, a \textit{cube} $Q\subset \rn$
always has finite edge length and all its edges parallel to the coordinate
axes and it is not necessary to be closed or open.
Moreover, for any $f\in L_{\loc}^1(\rn)$
and any bounded measurable set $E\subset\rn$ with $|E|>0$, let
$$f_E:=\fint_{E}f(x)\,dx:=\frac{1}{|E|}\int_E f(x)\,dx.$$
\begin{defn}\label{d-jnp}
Let $p\in[1,\infty)$ and $Q_0$ be a given cube of $\rn$.
The \emph{John--Nirenberg space  $JN_p(Q_0)$}
is defined to be the set of all $f\in L^1(Q_0)$ such that
\begin{align*}
\|f\|_{JN_p(Q_0)}:=
\sup\lf\{\sum_{j}\lf|Q_{j}\r|\lf[\fint_{Q_j}
\lf|f(x)-f_{Q_{j}}\r|\,dx\r]^{p}\r\}^{\frac{1}{p}}<\infty,
\end{align*}
where the supremum is taken over all collections
$\{Q_j\}_j$ of interior pairwise disjoint subcubes of $Q_0$.
\end{defn}

Since \cite{JN}, the John--Nirenberg space has attracted a lot of attention.
For instance, Dafni et al. \cite{DHKY} showed the nontriviality of $JN_p(Q_0)$
and introduced a Hardy-kind space $HK_{p'}(Q_0)$ which proves
the predual space of the space $JN_p(Q_0)$, where
$p\in(1,\fz)$ and $\frac{1}{p}+\frac{1}{p'}=1$.
Berkovits et al. \cite{bkm2016} applied the dyadic variant of $JN_p(Q_0)$
in the study of self-improving
properties of some Poincar\'e-type inequalities.
Very recently, Dom\'inguez and Milman \cite{DM} introduced
and studied sparse Brudnyi and John--Nirenberg spaces. We refer the reader to
\cite{ABKY, bkm2016, jyyz2, M, MM, SXY, TYY19, TYY2, tyy3}
for more studies on John--Nirenberg-type spaces.
Although there exist many studies on the John--Nirenberg space
and its predual space, it is still a challenging
and open question to obtain the boundedness
of some important operators, such as
the Hardy--Littlewood maximal operator, the
Calder\'on--Zygmund operator, and the fractional integral,
on these spaces. The main results of this article may
shed some light on this.

In what follows, for any $\ell\in(0,\fz)$, let
$\Pi_\ell(\rn)$ be the class of all collections of
interior pairwise disjoint subcubes
$\{Q_j\}_j$ of $\rn$ with edge length $\ell$;
for any $s\in\zz_+:=\{0,1,2,\ldots\}$,
let $\mathcal{P}_s(\rn)$ denote the set of
all polynomials of total degree not greater than $s$ on $\rn$;
moreover, for any
$\gamma:=(\gamma_1,\ldots,\gamma_n)\in\zz_+^n:=(\zz_+)^n$
and $x:=(x_1,\ldots,x_n)\in\rn$,
let $|\gamma|:=\gamma_1+\cdots+\gamma_n$ and
$x^{\gamma}:=x_1^{\gamma_1}\cdots x_n^{\gamma_n}$.
\begin{defn}\label{Defin.jncc}
Let $p\in[1,\infty]$, $q\in[1,\infty)$, $s\in\zz_+$, and $\alpha\in\rr$.
The \emph{special John--Nirenberg--Campanato space via congruent cubes}
(for short, \textit{congruent} JNC \textit{space}),
$JN_{(p,q,s)_\alpha}^{\mathrm{con}}(\rn)$, is defined to be the set of all
$f\in L^q_{\mathrm{loc}}(\rn)$ such that
\begin{align*}
\|f\|_{JN_{(p,q,s)_\alpha}^{\mathrm{con}}(\rn)}
&:=
\begin{cases}
\displaystyle
\sup_{\substack{\ell\in(0,\fz)\\ \{Q_j\}_j\in\Pi_{\ell}(\rn)}}
\lf[\sum_{j}\lf|Q_{j}\r|\lf\{\lf|Q_{j}
\r|^{-\alpha}\phantom{\lf[\int_{Q_j}\r]^{\frac{1}{q}}}\r.\r.\\
\displaystyle
\quad\lf.\lf.\times\lf[\fint_{Q_{j}}
\lf|f(x)-P_{Q_j}^{(s)}(f)(x)\r|^{q}\,dx
\r]^{\frac{1}{q}}\r\}^{p} \r]^{\frac{1}{p}}
&\text{if}\ p\in[1,\fz), \\
\displaystyle
\sup_{{\rm cube\ }Q\subset \rn}|Q|^{-\alpha}\lf[\fint_{Q}
\lf|f(x)-P_{Q}^{(s)}(f)(x)\r|^{q}\,dx\r]^{\frac{1}{q}}
&\text{if}\ p=\fz
\end{cases}
\end{align*}
is finite, here and thereafter, for any cube (or ball) $Q\subset \rn$,
$P_{Q}^{(s)}(f)$ denotes the unique polynomial
of total degree not greater than $s$ such that,
for any $\gamma\in\zz_+^n$ with $|\gamma|\leq s$,
\begin{align}\label{pq}
\int_{Q}\lf[f(x)-P_{Q}^{(s)}(f)(x)\r]x^{\gamma}\,dx=0.
\end{align}
\end{defn}

The space $JN_{(\fz,q,s)_\alpha}^{\mathrm{con}}(\rn)$
is just the well-known Campanato space
$\mathcal{C}_{\alpha,q,s}(\rn)$ which was introduced
by Campanato \cite{C} and when $\alpha=0$ coincides with the space
$\mathrm{BMO}\,(\rn)$.
Moreover, in article \cite{jtyyz1}, we introduced a
Hardy-kind space $HK_{(p,q,s)_\alpha}^{\mathrm{con}}(\rn)$
which proves the predual of $JN_{(p',q',s)_\alpha}^{\mathrm{con}}(\rn)$
with $p\in(1,\fz)$
(see, for instance, \cite[Theorem 4.10]{jtyyz1}),
here and thereafter, $p'$ and $q'$ denote, respectively,
the conjugate indexes of $p$ and $q$,
namely, $\frac{1}{p}+\frac{1}{p'}=1=\frac{1}{q}+\frac{1}{q'}$.

The main propose of this article is to
extend the Hardy--Littlewood--Sobolev theorem to $JN_{(p,q,s)_\alpha}^{\mathrm{con}}(\rn)$,
the special John--Nirenberg--Campanato space
via congruent cubes,
and the Hardy-kind space
$HK_{(p,q,s)_\alpha}^{\mathrm{con}}(\rn)$, where $p\in[1,\infty]$, $q\in[1,\infty)$,
$s\in\mathbb{Z}_+:=\mathbb{N}\cup\{0\}$, and
$\alpha\in\mathbb{R}$. Let $\bz\in (0,1)$.
We first introduce the (generalized) fractional integral, denoted by $I_{\beta}$,
in Definition \ref{Def-I} below.
Then we find an reasonable version of the fractional integral,
denoted by $\widetilde{I}_{\beta}$, on
$JN_{(p,q,s)_\alpha}^{\mathrm{con}}(\mathbb{R}^n)$.
To this end, we introduce
the vanishing moments up to order $s$ of $I_{\beta}$
in Definition \ref{Def-I-s} below.
Moreover, we prove that $\widetilde{I}_{\beta}$ is bounded
from $JN_{(p,q,s)_\alpha}^{\mathrm{con}}(\mathbb{R}^n)$
to $JN_{(p,q,s)_{\alpha+\beta/n}}^{\mathrm{con}}(\mathbb{R}^n)$
if and only if $I_{\beta}$ has the vanishing moments up to order $s$;
see Theorem \ref{Iw-JN-bound} below. To this end,
we find an equivalent version of
$I_{\beta}$ having the vanishing moments up to order $s$
in Proposition \ref{I-a-x} below.
The obtained result is new even for the Campanato space
$\mathcal{C}_{\alpha,q,s}(\mathbb{R}^n)=JN_{(\fz,q,s)_\alpha}^{\mathrm{con}}(\rn)$ with
$s\in\mathbb{N}:=\{1,2,\ldots\}$.
Furthermore, we show that $I_{\beta}$
can be extended to a unique continuous linear operator
from $HK_{(p,q,s)_{\alpha+\beta/n}}^{\mathrm{con}}(\mathbb{R}^n)$
to $HK_{(p,q,s)_{\alpha}}^{\mathrm{con}}(\mathbb{R}^n)$
if and only if $I_{\beta}$ has the vanishing
moments up to order $s$; see Theorem \ref{I-Bounded-HK} below.
In the proof of Theorem \ref{I-Bounded-HK}, we skillfully use molecules of
$HK_{(p,q,s)_\alpha}^{\mathrm{con}}(\mathbb{R}^n)$
and a criterion for the boundedness of linear operators on
$HK_{(p,q,s)_\alpha}^{\mathrm{con}}(\mathbb{R}^n)$,
obtained in \cite[Theorem 3.16]{jtyyz2}, to overcome the difficulty caused by
the fact that $\|\cdot\|_{HK_{(p,q,s)_{\alpha}}^{\mathrm{con}}(\mathbb{R}^n)}$
is not concave.

The remainder of this article is organized as follows.

In Section \ref{S-I-JN},
we first find a reasonable version $\widetilde{I}_{\beta}$
of the (generalized) fractional integral $I_{\beta}$
on $JN_{(p,q,s)_\alpha}^{\mathrm{con}}(\rn)$
via borrowing some ideas from \cite[Section 5]{N10}.
Moreover, we prove that $\widetilde{I}_{\beta}$ is bounded
from $JN_{(p,q,s)_\alpha}^{\mathrm{con}}(\mathbb{R}^n)$
to $JN_{(p,q,s)_{\alpha+\beta/n}}^{\mathrm{con}}(\mathbb{R}^n)$
if and only if $I_{\beta}$ has the vanishing moments up to order $s$
in Theorem \ref{Iw-JN-bound} below,
which is new even for the
Campanato space $\mathcal{C}_{\alpha,q,s}(\rn)=JN_{(\fz,q,s)_\alpha}^{\mathrm{con}}(\rn)$
with $s\in\nn$.

Section \ref{S-C-Z-HK} consists of two subsections.
In Subsection \ref{Hardy-type}, we first recall
the notions of the Hardy-kind space
$HK_{(p,q,s)_{\alpha}}^{\mathrm{con}}(\rn)$
and the molecule of $HK_{(p,q,s)_{\alpha}}^{\mathrm{con}}(\rn)$.
In Subsection \ref{S-I-HK}, via using the
molecule of $HK_{(p,q,s)_{\alpha+\beta/n}}^{\mathrm{con}}(\rn)$,
the dual relation $$\lf(HK_{(p,q,s)_{\alpha+\beta/n}}^{\mathrm{con}}(\rn)\r)^*
=JN_{(p',q',s)_{\alpha+\beta/n}}^{\mathrm{con}}(\rn),$$
and a criterion for the boundedness of linear operators on
$HK_{(p,q,s)_\alpha}^{\mathrm{con}}(\mathbb{R}^n)$,
obtained in \cite{jtyyz2}, we prove that $I_{\beta}$
can be extended to a unique continuous linear operator
from $HK_{(p,q,s)_{\alpha+\beta/n}}^{\mathrm{con}}(\mathbb{R}^n)$
to $HK_{(p,q,s)_{\alpha}}^{\mathrm{con}}(\mathbb{R}^n)$
if and only if $I_{\beta}$ has the vanishing
moments up to order $s$;
see Theorem \ref{I-Bounded-HK} below.

Finally, we make some conventions on notation. Let
$\nn:=\{1,2,\ldots\}$ and $\zz_+:=\nn\cup\{0\}$.
For any $s\in\zz_+$ and any ball $B\subset \rn$,
we use $\mathcal{P}_s(B)$
[resp., $\mathcal{P}_s(\rn)$] to denote the set of all
polynomials of total degree not greater than $s$ on $B$ (resp., $\rn$).
We always denote by $C$ a \emph{positive constant}
which is independent of the main parameters,
but it may vary from line to line.
The symbol $f\lesssim g$ means that $f\le Cg$.
If $f\lesssim g$ and $g\lesssim f$,
we then write $f\sim g$. If $f\le Cg$ and $g=h$
or $g\le h$, we then write $f\ls g\sim h$
or $f\ls g\ls h$, \emph{rather than} $f\ls g=h$ or $f\ls g\le h$.
We use $\mathbf{1}_E$ to denote the
characteristic function of a measurable $E\subset \rn$,
and $\mathbf{0}$ to denote the \emph{origin} of $\rn$.
For any $x\in\rn$ and $r\in(0,\fz)$,
we denote by $B(x,r):=\{y\in\rn: |y-x|<r\}$
the ball with center $x$ and radius $r$.
Moreover, for any $\lambda\in(0,\infty)$ and any
ball $B:=B(x,r)\subset\rn$ with $x\in\rn$ and
$r\in(0,\fz)$, let $\lambda B:=B(x,\lambda r)$.
We use $Q_z(r)$ to denote the cube with center
$z\in\rn$ and edge length $r\in(0,\fz)$. Finally, for any $p\in[1,\infty]$,
we denote by $p'$ its \emph{conjugate index},
namely, $\frac{1}{p}+\frac{1}{p'}=1$.

\noindent\\[4mm]

\section{Boundedness of fractional integrals
on $JN_{(p,q,s)_\alpha}^{\mathrm{con}}(\rn)$\label{S-I-JN}}

In this section,
we first introduce the (generalized) fractional integral $I_{\beta}$
and then find a reasonable version $\widetilde{I}_{\beta}$
of  $I_{\beta}$ on
$JN_{(p,q,s)_\alpha}^{\mathrm{con}}(\rn)$.
To this end, we give the notion of the
vanishing moments up to order $s$ in Definition \ref{Def-I-s} below.
Moreover, we prove that
$\widetilde{I}_{\beta}$ is bounded
from $JN_{(p,q,s)_\alpha}^{\mathrm{con}}(\rn)$
to $JN_{(p,q,s)_{\alpha+\beta/n}}^{\mathrm{con}}(\rn)$
if and only if $I_{\beta}$ has vanishing moments up to order $s$.

We begin with the  notion of the (generalized) fractional
integral $I_{\beta}$. In what follows,
for any $\gamma=(\gamma_1,\ldots,\gamma_n)\in \zz_+^n$,
any $\gamma$-order differentiable function $G$
on $\rn\times \rn$, and any $(x,y)\in \rn\times \rn$, let
$$
\partial_{(1)}^{\gamma}G(x,y):=\frac{\partial^{|\gamma|}}
{\partial x_1^{\gamma_1}\cdots\partial x_n^{\gamma_n}}G(x,y)
\quad\text{and}\quad\partial_{(2)}^{\gamma}G(x,y):=\frac{\partial^{|\gamma|}}
{\partial y_1^{\gamma_1}\cdots\partial y_n^{\gamma_n}}G(x,y).
$$
\begin{defn}\label{I-k}
Let $s\in\zz_+$, $\dz\in(0,1]$, and $\beta\in(0,n)$.
A measurable function $k_{\beta}$
on $\rn\times \rn\setminus\{(x,x):\ x\in\rn\}$ is called an
\textit{$s$-order fractional kernel with regularity $\dz$} if
there exists a positive constant $C$ such that
\begin{itemize}
\item[\rm (i)]
for any $\gamma\in\zz_+^n$ with $|\gamma|\leq s$,
and any $x$, $y\in\rn$ with $x\neq y$,
\begin{align}\label{size-i'}
\lf|\partial_{(2)}^{\gamma}k_{\beta}(x,y)\r|
\le C\frac{1}{|x-y|^{n+|\gamma|-\beta}};
\end{align}

\item[\rm (ii)]
for any $\gamma\in\zz_+^n$ with $|\gamma|\leq s$,
and any $x$, $y$, $z\in\rn$ with $x\neq y$ and $|x-y|\ge2|y-z|$,
\begin{align}\label{regular-i'}
\lf|\partial_{(2)}^{\gamma}k_{\beta}(x,y)
-\partial_{(2)}^{\gamma}k_{\beta}(x,z)\r|
\le C\frac{|y-z|^\dz}{|x-y|^{n+|\gamma|+\dz-\beta}}.
\end{align}
\end{itemize}
\end{defn}

\begin{defn}\label{Def-I}
Let $s\in\zz_+$, $\dz\in(0,1]$, $\beta\in(0,n)$,
and $k_{\beta}$ be an $s$-order fractional kernel with regularity $\dz$.
The \emph{(generalized) fractional integral $I_{\beta}$
with kernel $k_{\beta}$} is defined by setting,
for any suitable function $f$ on $\rn$,
and almost every $x\in\rn$,
\begin{align}\label{2.2x}
I_{\beta}(f)(x)
:=\int_{\rn}k_{\beta}(x,y)f(y)\,dy.
\end{align}
\end{defn}

\begin{rem}\label{rem-I'}
In Definition \ref{Def-I}, if $\dz:=1$, $\beta\in(0,n)$,
and $k_{\beta}:=\frac{1}{|x-y|^{n-\beta}}$,
then it is easy to show that $k_{\beta}$
satisfies \eqref{size-i'} and \eqref{regular-i'},
and hence $I_\beta$  in  this case coincides with
$\mathcal{I}_{\beta}$ in \eqref{cla-I}.
\end{rem}

In what follows, for any measurable function $f$ on $\rn$, we define
its \textit{support} $\supp\,(f)$ by setting
$$\supp\,(f):=\{x\in\rn:\ f(x)\neq0\}.$$
Inspired by \cite[Definition 9.4]{Bo2003},
we give the following notion.

\begin{defn}\label{Def-I-s}
Let $s\in\zz_+$, $\dz\in(0,1]$, $\beta\in(0,n)$,
and $k_{\beta}$ be an $s$-order fractional kernel with regularity $\dz$
as in Definition \ref{I-k}.
The fractional integral $I_{\beta}$ with kernel $k_{\beta}$
is said to have the \emph{vanishing moments up to
order $s$} if, for any $a\in L^2(\rn)$ having bounded support
and satisfying that, for any $\gamma\in\zz_+^n$ with $|\gamma|\leq s$,
$\int_{\rn} a(x)x^\gamma\,dx=0$, it holds true that
\begin{align}\label{I-x-gam}
\int_{\rn} I_{\beta}(a)(x)x^\gamma\,dx=0.
\end{align}
\end{defn}

\begin{rem}\label{Def-I'}
It was proved in \cite[p.\,104]{TW1980}
(see also \cite[p.\,107]{L} and \cite[Lemma 3.2]{ans2021}) that, if $\bz\in (0,1)$,
then $\mathcal{I}_{\beta}$ in \eqref{cla-I} has the vanishing moments up to
order $s$ for any $s\in\zz_+$,
which shows that \eqref{I-x-gam} is a reasonable assumption
when $\bz\in (0,\dz)$ and $\dz\in (0,1]$. However, when $\beta\in[1,n)$, 
it is still unknown whether or not $\mathcal{I}_{\beta}$ in \eqref{cla-I} 
has the vanishing moments up to order $s$.
\end{rem}

We now show that $I_\beta$ in Definition \ref{Def-I-s}
induces a well-defined operator $\widetilde{I}_{\beta,B_0}$
on $JN_{(p,q,s)_\alpha}^{\mathrm{con}}(\rn)$ via borrowing some ideas
from \cite[Section 5]{N10}.

\begin{defn}\label{kw}
Let $s\in\zz_+$, $\dz\in(0,1]$, $\beta\in(0,n)$,
and $k_{\beta}$ be an $s$-order fractional
kernel with regularity $\dz$.
The \textit{adjoint kernel} $\widetilde{k}_{\beta}$ of $k_{\beta}$
is defined by setting,
for any $x$, $y\in \rn$ with $x\neq y$,
\begin{align}\label{k-w}
\widetilde{k}_{\beta}(x,y):=k_{\beta}(y,x).
\end{align}
Moreover, the operator in \eqref{2.2x} with
$k_{\beta}$ replaced by $\widetilde{k}_{\beta}$ is also
called the (generalized) fractional integral.
\end{defn}

\begin{rem}\label{rem-kw}
In Definition \ref{kw},
by \eqref{size-i'} and \eqref{regular-i'},
we find that there exists a positive constant $C$ such that
\begin{itemize}
\item[\rm (i)]
for any $\gamma\in\zz_+^n$ with $|\gamma|\leq s$,
and any $x$, $y\in\rn$ with $x\neq y$,
\begin{align}\label{size-i}
\lf|\partial_{(1)}^{\gamma}\widetilde{k}_{\beta}(x,y)\r|
\le C\frac{1}{|x-y|^{n+|\gamma|-\beta}};
\end{align}

\item[\rm (ii)]
for any $\gamma\in\zz_+^n$ with $|\gamma|\leq s$,
and any $x$, $y$, $\omega\in\rn$
with $x\neq y$ and $|x-y|\ge2|x-\omega|$,
\begin{align}\label{regular-i}
\lf|\partial_{(1)}^{\gamma}\widetilde{k}_{\beta}(x,y)
-\partial_{(1)}^{\gamma}\widetilde{k}_{\beta}(\omega,y)\r|
\le C\frac{|x-\omega|^\dz}{|x-y|^{n+|\gamma|+\dz-\beta}}.
\end{align}
\end{itemize}
Moreover, let $s=0$. Then $\widetilde{k}_{\beta}$ coincides
with the fractional kernel introduced in
\cite[Definition 4.1]{JA2005} over $\rn$.
\end{rem}

\begin{defn}\label{I-w}
Let $s\in\zz_+$, $\dz\in(0,1]$, $\beta\in(0,n)$, and $k_{\beta}$ be an
$s$-order fractional kernel with regularity $\dz$.
Let $\widetilde{k}_{\beta}$ be the adjoint kernel of $k_{\beta}$ as in \eqref{k-w},
and $B_0:=B(x_0,r_0)$ a given ball of $\rn$
with $x_0\in\rn$ and $r_0\in (0,\fz)$.
The \emph{modified fractional integral
$\widetilde{I}_{\beta,B_0}$ with kernel $\widetilde{k}_{\beta}$}
is defined by setting,
for any suitable function $f$
on $\rn$, and almost every $x\in \rn$,
\begin{align}\label{D-Iw}
&\widetilde{I}_{\beta,B_0}(f)(x)\noz\\
&\quad:=\int_{\rn}\lf[\widetilde{k}_{\beta}(x,y)
-\sum_{\{\gamma\in\zz_+^n:\ |\gamma|\leq s\}}
\frac{\partial_{(1)}^{\gamma}\widetilde{k}_{\beta}
(x_0,y)}{\gamma!}(x-x_0)^{\gamma}
\mathbf{1}_{\rn\setminus B_0}(y)\r]f(y)\,dy.
\end{align}
\end{defn}

\begin{rem}	
In Definition \ref{I-w}, let
$s=0$, $\dz=1$, $\beta\in(0,1)$, and
$I_\beta$ be the fractional integral with kernel $k_{\beta}$.
Then $I_\beta$ is just $\mathcal{I}_{\beta}$
as in \eqref{cla-I}, and $\widetilde{I}_{\beta,B_0}$
coincides with \cite[(5.1)]{N10} which was
defined on Campanato spaces over spaces of homogeneous
type with variable growth conditions.
\end{rem}

Next, we show that $\widetilde{I}_{\beta,B_0}$
is well defined on $L^q(\rn)$ with $q\in[1,\fz)$, and
coincides with the fractional integral $\widetilde{I}_{\beta}$,
with kernel $\widetilde{k}_{\beta}$, which is
defined by setting, for any $f\in L^q(\rn)$ with $q\in[1,\frac{n}{\beta})$, and almost every $x\in\rn$,
$$
\widetilde{I}_{\beta}(f)(x)
:=\int_{\rn}\widetilde{k}_{\beta}(x,y)f(y)\,dy,
$$
in the sense of modulo $\mathcal{P}_s(\rn)$. To this end, we need
the following two technical lemmas. The first lemma was
stated in \cite[p.\,250]{JA2005},
which can be easily obtained
by \eqref{size-i'} [or \eqref{size-i}] with $\gamma:=\mathbf{0}$,
and \cite[p.\,119, Theorem 1]{EMS1970}.

\begin{lem}\label{fractional}
Let $s\in\zz_+$, $\dz\in(0,1]$, $\beta\in(0,n)$,
$k_{\beta}$ be an $s$-order fractional kernel with regularity $\dz$,
$\widetilde{k}_{\beta}$ as in \eqref{k-w}, and
$I_\beta$ the fractional integral  with kernel $k_{\beta}$
(or $\widetilde{k}_{\beta}$).
Let $b\in[1,\frac{n}{\beta})$ and $\widetilde{b}\in(1,\infty)$ with
$\frac{1}{\widetilde{b}}=\frac{1}{b}-\frac{\beta}{n}$. Then
\begin{enumerate}
\item[\rm (i)]
for any given $g\in L^b(\rn)$,
$I_\beta(g)(x)$ is well defined for almost every $x\in\rn$;
\item[\rm (ii)]
$I_\beta$ is bounded from $L^b(\rn)$
to $L^{\widetilde{b}}(\rn)$, namely, there exists a
positive constant $C$ such that, for any $g\in L^b(\rn)$,
$$\lf\|I_\beta(g)\r\|_{L^{\widetilde{b}}(\rn)}\leq C\|g\|_{L^b(\rn)}.$$
\end{enumerate}
\end{lem}

The second lemma shows that, for any bounded
measurable set $E\subset \rn$ and any $f\in L_{\mathrm{loc}}^q(\rn)$
with $q\in(\frac{n}{\beta},\fz)$, $I_\beta(f\mathbf{1}_E)$
is bounded on $\rn$.

\begin{lem}\label{lem2.10x}
Let $s\in\zz_+$, $\dz\in(0,1]$, $\beta\in(0,n)$,
$k_{\beta}$ be an $s$-order fractional kernel with regularity $\dz$,
and $I_{\beta}$ a fractional integral with kernel $k_{\beta}$.
Let $q\in[1,\fz)$ and $E$ be a given bounded
measurable set of $\rn$. Then, for any $f\in L_{\mathrm{loc}}^{q}(\rn)$,
$I_\beta(f\mathbf{1}_E)$
is well defined almost everywhere on $\rn$.
Moreover, if $q\in(\frac{n}{\beta},\fz)$, then
$I_\beta(f\mathbf{1}_E)$ is a bounded function on $\rn$,
and hence $I_\beta(f\mathbf{1}_E)
\in L_{\mathrm{loc}}^{b}(\rn)$ for any given $b\in[1,\fz)$.	
\end{lem}

\begin{proof}
Let $q$ and $E$ be as in the present lemma.
Also, let $B(\mathbf{0},r_0)$ be a given ball of
$\rn$ containing $E$ with $r_0\in(0,\fz)$.
For any given $f\in L_{\mathrm{loc}}^q(\rn)$,
by the H\"older inequality, we find that $f\mathbf{1}_{E}\in L^1(\rn)$.
From this and Lemma \ref{fractional}(i), we deduce that
$I_\beta(f\mathbf{1}_E)$
is well defined almost everywhere on $\rn$.
Moreover, if $q\in(\frac{n}{\beta},\fz)$, then $q'(n-\beta)\in(0,n)$,
which, together with the fact that $|y|<r_0\leq|x-y|$ for any 
$x\in\rn\setminus\{\mathbf{0}\}$ and
$y\in B(\mathbf{0},r_0)\setminus B(x,r_0)$,
further implies that, for any $x\in\rn$,
\begin{align*}
&\int_{B(\mathbf{0},r_0)}
\frac{1}{|y|^{q'(n-\beta)}}\,dy
-\int_{B(\mathbf{0},r_0)}
\frac{1}{|x-y|^{q'(n-\beta)}}\,dy\\
&\quad=\int_{B(\mathbf{0},r_0)}
\frac{1}{|y|^{q'(n-\beta)}}\,dy
-\int_{B(x,r_0)}
\frac{1}{|y|^{q'(n-\beta)}}\,dy\\
&\quad=\int_{B(\mathbf{0},r_0)\setminus B(x,r_0)}
\frac{1}{|y|^{q'(n-\beta)}}\,dy
-\int_{B(x,r_0)\setminus B(\mathbf{0},r_0)}
\frac{1}{|y|^{q'(n-\beta)}}\,dy\\
&\quad=\int_{B(\mathbf{0},r_0)\setminus B(x,r_0)}
\frac{1}{|y|^{q'(n-\beta)}}\,dy
-\int_{B(-x,r_0)\setminus B(\mathbf{0},r_0)}
\frac{1}{|-y|^{q'(n-\beta)}}\,dy\\
&\quad=\int_{B(\mathbf{0},r_0)\setminus B(x,r_0)}
\lf(\frac{1}{|y|^{q'(n-\beta)}}-\frac{1}{|x-y|^{q'(n-\beta)}}\r)\,dy
\geq0.
\end{align*}
Using this, \eqref{size-i'} with $\gamma:=\mathbf{0}$,
the H\"older inequality, and $q'(n-\beta)\in(0,n)$,
we conclude that, for any $x\in\rn$,
\begin{align*}
\lf|I_\beta(f\mathbf{1}_E)(x)\r|
&\ls\int_{E}\frac{|f(y)|}{|x-y|^{n-\beta}}\,dy
\ls \|f\|_{L^q(E)}\lf[\int_{E}
\frac{1}{|x-y|^{q'(n-\beta)}}\,dy\r]^{\frac{1}{q'}}\\
&\ls \|f\|_{L^q(E)}\lf[\int_{B(\mathbf{0},r_0)}
\frac{1}{|x-y|^{q'(n-\beta)}}\,dy\r]^{\frac{1}{q'}}\\
&\ls \|f\|_{L^q(E)}\lf[\int_{B(\mathbf{0},r_0)}
\frac{1}{|y|^{q'(n-\beta)}}\,dy\r]^{\frac{1}{q'}}
\ls 1.
\end{align*}
This finishes the proof of Lemma \ref{lem2.10x}.
\end{proof}

\begin{pro}\label{prop-2.18x}
Let $s\in\zz_+$, $\dz\in(0,1]$, $\beta\in(0,\dz)$,
and $k_{\beta}$ be an $s$-order fractional kernel with regularity $\dz$.
Let $\widetilde{k}_{\beta}$ be as in \eqref{k-w}
and $B_0:=B(x_0,r_0)$ a given ball of $\rn$
with $x_0\in\rn$ and $r_0\in (0,\fz)$.
Let $\widetilde{I}_{\beta}$ be the fractional integral
with kernel $\widetilde{k}_{\beta}$.
Then, for any $f\in L^q(\rn)$ with $q\in[1,\fz)$,
$\widetilde{I}_{\beta,B_0}(f)$ in \eqref{D-Iw}
is well defined almost everywhere on $\rn$
and, if $q\in[1,\frac{n}{\beta})$, then
$\widetilde{I}_{\beta}(f)-\widetilde{I}_{\beta,B_0}(f)\in \mathcal{P}_s(\rn)$
after changing values on a set of measure zero.	
\end{pro}

\begin{proof}
Let $s$, $\dz$, $k_{\beta}$, $\widetilde{k}_{\beta}$,
$\widetilde{I}_{\beta}$, and $B_0:=B(x_0,r_0)$
with $x_0\in\rn$ and $r_0\in(0,\fz)$ be as in the present proposition.
We first show that, for any $f\in L^q(\rn)$ with $q\in[1,\fz)$,
$\widetilde{I}_{\beta,B_0}(f)$ exists almost everywhere on $\rn$.
Indeed, for any $x\in\rn$, let $\widetilde{B}_{(x)}:=B(x_0,2|x-x_0|+2r_0)$. Then,
by \eqref{size-i}, the Taylor remainder theorem,
the H\"older inequality,
Lemma \ref{lem2.10x}, and $\beta\in(0,\dz)$, we find that,
for any $f\in L^q(\rn)$ with $q\in[1,\fz)$, and almost every $x\in\rn$,
there exists an $\widetilde{x}\in
\{\theta x+(1-\theta) x_0\in\rn:\ \theta\in(0,1)\}$ such that
\begin{align*}
\lf|\widetilde{I}_{\beta,B_0}(f)(x)\r|
&\leq\lf|\int_{\widetilde{B}_{(x)}}
\widetilde{k}_{\beta}(x,y)f(y)\,dy\r|\\
&\quad+\int_{\widetilde{B}_{(x)}\setminus B_0}\lf|
\sum_{\{\gamma\in\zz_+^n:\ |\gamma|\leq s\}}
\frac{\partial_{(1)}^{\gamma}\widetilde{k}_{\beta}
(x_0,y)}{\gamma!}(x-x_0)^{\gamma}\r||f(y)|\,dy\\
&\quad+\int_{\rn\setminus \widetilde{B}_{(x)}}\lf|\widetilde{k}_{\beta}(x,y)
-\sum_{\{\gamma\in\zz_+^n:\ |\gamma|\leq s\}}
\frac{\partial_{(1)}^{\gamma}\widetilde{k}_{\beta}
(x_0,y)}{\gamma!}(x-x_0)^{\gamma}\r||f(y)|\,dy\\
&\ls\lf|\widetilde{I}_{\beta}(f\mathbf{1}_{\widetilde{B}_{(x)}})(x)\r|
+\sum_{\{\gamma\in\zz_+^n:\ |\gamma|\leq s\}}
\int_{\widetilde{B}_{(x)}\setminus B_0}\frac{|x-x_0|^{|\gamma|}|f(y)|}
{|x_0-y|^{n+|\gamma|-\beta}}\,dy\\
&\quad+\int_{\rn\setminus \widetilde{B}_{(x)}}\lf|
\sum_{\{\gamma\in\zz_+^n:\ |\gamma|= s\}}
\frac{\partial_{(1)}^{\gamma}\widetilde{k}_{\beta}(\widetilde{x},y)
-\partial_{(1)}^{\gamma}\widetilde{k}_{\beta}
(x_0,y)}{\gamma!}(x-x_0)^{\gamma}\r||f(y)|\,dy\\
&\ls\lf|\widetilde{I}_{\beta}(f\mathbf{1}_{\widetilde{B}_{(x)}})(x)\r|
+\sum_{\{\gamma\in\zz_+^n:\ |\gamma|\leq s\}}
|x-x_0|^{|\gamma|}
r_0^{-(n+|\gamma|-\beta)}
\|f\|_{L^1(\widetilde{B}_{(x)}\setminus B_0)}\\
&\quad+|\widetilde{x}-x_0|^\dz|x-x_0|^{s}\int_{\rn\setminus \widetilde{B}_{(x)}}
\frac{|f(y)|}{|x_0-y|^{n+s+\dz-\beta}}\,dy\\
&\quad\ls\lf|\widetilde{I}_{\beta}(f\mathbf{1}_{\widetilde{B}_{(x)}})(x)\r|+
\sum_{\{\gamma\in\zz_+^n:\ |\gamma|\leq s\}}|x-x_0|^{|\gamma|}
r_0^{-(n+|\gamma|-\beta)}
\|f\|_{L^1(\widetilde{B}_{(x)}\setminus B_0)}\\
&\quad+|\widetilde{x}-x_0|^\dz|x-x_0|^{s}
\lf\|\frac{1}{|x_0-\cdot|^{n+s+\dz-\beta}}
\r\|_{L^{q'}(\rn\setminus \widetilde{B}_{(x)})}\|f\|_{L^q(\rn)}\\
&<\fz,
\end{align*}
where, in the third step, we used \eqref{regular-i}
together with $|y-x_0|> 2|\widetilde{x}-x_0|$
for any $y\in \rn\setminus \widetilde{B}_{(x)}$.
This implies that
$\widetilde{I}_{\beta,B_0}(f)$ is well defined almost everywhere on $\rn$.

Moreover, using \eqref{size-i} and the H\"older inequality, we find that,
for any $f\in L^q(\rn)$ with $q\in[1,\frac{n}{\beta})$,
and $\gamma\in\zz_+^n$ with $|\gamma|\leq s$,
\begin{align*}
\int_{\rn\setminus B_0}
\lf|\partial_{(1)}^{\gamma}\widetilde{k}_{\beta}(x_0,y)f(y)\r|\,dy
&\ls\int_{\rn\setminus B_0}\frac{|f(y)|}{|x_0-y|^{n+|\gamma|-\beta}}\,dy\\
&\ls\|f\|_{L^q(\rn)}
\lf\|\frac{1}{|x_0-\cdot|^{n+|\gamma|-\beta}}
\r\|_{L^{q'}(\rn\setminus B_0)}
<\fz,
\end{align*}
where, in the last step, we used $q'(n+|\gamma|-\beta)\in(n,\fz]$.
This shows that, for any $f\in L^q(\rn)$
with $q\in[1,\frac{n}{\beta})$, and almost every $x\in\rn$,
\begin{align*}
&\widetilde{I}_{\beta}(f)(x)-\widetilde{I}_{\beta,B_0}(f)(x)\\
&\quad=\int_{\rn}
\lf[\sum_{\{\gamma\in\zz_+^n:\ |\gamma|\leq s\}}
\frac{\partial_{(1)}^{\gamma}
\widetilde{k}_{\beta}(x_0,y)}{\gamma!}(x-x_0)^\gamma
\mathbf{1}_{\rn\setminus B_0}(y)\r]f(y)\,dy\\
&\quad=\sum_{\{\gamma\in\zz_+^n:\ |\gamma|\leq s\}}
\frac{(x-x_0)^\gamma}{\gamma!}\int_{\rn\setminus B_0}
\partial_{(1)}^{\gamma}\widetilde{k}_{\beta}(x_0,y)f(y)\,dy,
\end{align*}
which further implies that
$\widetilde{I}_{\beta}(f)(x)-\widetilde{I}_{\beta,B_0}(f)(x)
\in \mathcal{P}_s(\rn)$
after changing values on a set of measure zero.
This finishes the proof of Proposition \ref{prop-2.18x}.
\end{proof}

Now, we show that
$\widetilde{I}_{\beta,B_0}$ in Definition \ref{I-w} is well defined
on $JN_{(p,q,s)_\alpha}^{\mathrm{con}}(\rn)$.
To this end, we first give several technical lemmas.
The following Lemmas \ref{int-B-P} and \ref{I-JN} are just, respectively,
\cite[Lemmas 2.14 and 2.21]{jtyyz2}, which play important roles
in the proofs of the main results of this article.

\begin{lem}\label{int-B-P}
Let $s\in\zz_+$ and $f$ be a measurable function on $\rn\times\rn$ such that, for any $y\in \rn$,
$f(\cdot,y)\in\mathcal{P}_s(\rn)$.
If, for almost every $x\in \rn$ (resp., $x\in B$),
\begin{align*}
P(x):=\int_{\rn}f(x,y)\,dy
\end{align*}
is finite, then $P\in \mathcal{P}_s(\rn)$ [resp., $P\in \mathcal{P}_s(B)$]
after changing values on a set of measure zero.
\end{lem}

\begin{lem}\label{I-JN}
Let $p\in[1,\fz]$, $q\in(1,\fz)$, $s\in\zz_+$, $\lambda\in(s,\fz)$,
and $\alpha\in (-\fz,\frac{1}{p}+\frac{\lz}{n})$.
Then there exists a positive constant $C$ such that,
for any $f\in JN_{(p,q,s)_\alpha}^{\mathrm{con}}(\rn)$ and
any ball $B(x,r)$ with $x\in\rn$ and $r\in(0,\fz)$,
\begin{align}\label{JN-I-E}
&\int_{\rn\setminus B(x,r)}\frac{|f(y)
-P_{B(x,r)}^{(s)}(f)(y)|}{|x-y|^{n+\lz}}\,dy\noz\\
&\quad\leq C\sum_{k\in\nn}\lf(2^kr\r)^{-\lz}
\lf[\fint_{2^{k}B(x,r)}\lf|f(y)
-P_{2^{k}B(x,r)}^{(s)}(f)(y)\r|^q\,dy\r]^{\frac{1}{q}}
\end{align}
and hence
\begin{align*}
\int_{\rn\setminus B(x,r)}\frac{|f(y)
-P_{B(x,r)}^{(s)}(f)(y)|}{|x-y|^{n+\lz}}\,dy
\leq Cr^{-\frac{n}{p}-\lz+\alpha n}
\|f\|_{JN_{(p,q,s)_\alpha}^{\mathrm{con}}(\rn)}.
\end{align*}
\end{lem}

Next, we give an equivalent expression of \eqref{I-x-gam}
in Proposition \ref{I-a-x} below. To this end,
we first establish a crucial lemma.
In what follows, for any $s\in\zz_+$, any $q\in[1,\infty]$,
and any measurable subset $E\subset\rn$,
the \emph{space $L_s^q(E)$} is defined by setting
\begin{align*}
L_s^q(E):=\lf\{f\in L^q(E):\ \int_{E}f(x)x^{\gamma}\,dx=0
\text{ for any } \gamma\in \zz_+^n\text{ with } |\gamma|\leq s\r\}.
\end{align*}

\begin{lem}\label{I-a-x-lem}
Let $s\in\zz_+$, $\dz\in(0,1]$,
$\beta\in(0,\dz)$, $k_{\beta}$ be an
$s$-order fractional kernel with regularity $\dz$,
and $I_\beta$ the
fractional integral with kernel $k_{\beta}$.
Let $\widetilde{k}_{\beta}$ be as in \eqref{k-w},
$B_0:=B(x_0,r_0)$ a given ball of $\rn$ with $x_0\in\rn$
and $r_0\in(0,\fz)$, and
$\widetilde{I}_{\beta,B_0}$ as in \eqref{D-Iw}
with kernel $\widetilde{k}_{\beta}$.
Then, for any $\nu:=(\nu_1,\ldots,\nu_n)\in\zz_+^n$ with $|\nu|\leq s$,
\begin{itemize}
\item[\rm (i)]
$\widetilde{I}_{\beta,B_0}(y^{\nu})$ is well defined almost
everywhere on $\rn$,
and $\widetilde{I}_{\beta,B_0}(y^{\nu})\in L^b_{\mathrm{loc}}(\rn)$
for any given $b\in[1,\fz)$, here and thereafter, for any
$y:=(y_1,\ldots,y_n)\in\rn$, $y^{\nu}:=y_1^{\nu_1}\cdots y_n^{\nu_n}$;

\item[\rm (ii)]
for any ball $B_1\subset\rn$,
\begin{align*}
\widetilde{I}_{\beta,B_0}(y^{\nu})
-\widetilde{I}_{\beta,B_1}(y^{\nu})\in \mathcal{P}_s(\rn)
\end{align*}
after changing values
on a set of measure zero;

\item[\rm (iii)]
for any $q\in(1,\fz)$ and $a\in L^q_{s}(\rn)$
with bounded support,
\begin{align}\label{0-0}
\int_{\rn} I_{\beta}(a)(x)x^{\nu}\,dx
=\int_{\rn}a(x)\widetilde{I}_{\beta,B_0}(y^{\nu})(x)\,dx.
\end{align}
\end{itemize}	
\end{lem}

\begin{proof}
Let $s$, $\dz$, $\beta$, $k_{\beta}$, $\widetilde{k}_{\beta}$,
$I_{\beta}$, $B_0:=B(x_0,r_0)$ with $x_0\in\rn$
and $r_0\in(0,\fz)$, and $\widetilde{I}_{\beta,B_0}$
be as in the present lemma.
We first prove (i). Indeed, let $\widetilde{B}$ be any given ball of $\rn$,
and $$R:=\sup\lf\{2|x-x_0|+2|x_0|+r_0+1:\ x\in \widetilde{B}\r\}.$$
It is easy to show that $(B_0\cup \widetilde{B})\subset B(x_0,R)$
and, for any $x\in \widetilde{B}$ and $y\in \rn\setminus B(x_0,R)$,
\begin{align*}
|y|\geq|y-x_0|-|x_0|\geq1,\quad |y-x_0|\geq2|x-x_0|,
\quad\text{and}\quad|y-x_0|\sim |y|.
\end{align*}
From this, the Taylor remainder theorem, \eqref{size-i},
and $\beta\in (0,\dz)$, we deduce that,
for any $\nu\in\zz_+^n$ with $|\nu|\leq s$,
and almost every $x\in \widetilde{B}$, there exists an
$\widetilde{x}\in\{\theta x+(1-\theta)x_0:\ \theta\in(0,1)\}$ such that
\begin{align*}
&\lf|\widetilde{I}_{\beta,B_0}(y^{\nu})(x)\r|\\
&\quad\leq\lf|\int_{B(x_0,R)}
\lf[\widetilde{k}_{\beta}(x,y)
-\sum_{\{\gamma\in\zz_+^n:\ |\gamma|\leq s\}}
\frac{\partial_{(1)}^{\gamma}
\widetilde{k}_{\beta}(x_0,y)}{\gamma!}(x-x_0)^\gamma
\mathbf1_{\rn\setminus B_0}(y)\r]y^{\nu}\,dy\r|\\
&\qquad+\int_{\rn\setminus B(x_0,R)}
\lf|\widetilde{k}_{\beta}(x,y)-\sum_{\{\gamma\in\zz_+^n:\ |\gamma|\leq s\}}
\frac{\partial_{(1)}^{\gamma}\widetilde{k}_{\beta}(x_0,y)}
{\gamma!}(x-x_0)^\gamma\r||y|^{|\nu|}\,dy\\
&\quad\leq\lf|\int_{B(x_0,R)}
\widetilde{k}_{\beta}(x,y)y^{\nu}\,dy\r|
+\sum_{\{\gamma\in\zz_+^n:\ |\gamma|\leq s\}}
\int_{B(x_0,R)\setminus B_0}\frac{
|\partial_{(1)}^{\gamma}\widetilde{k}_{\beta}
(x_0,y)||x-x_0|^{|\gamma|}|y|^{|\nu|}}{\gamma!}\,dy\\
&\qquad+\int_{\rn\setminus B(x_0,R)}
\lf|\sum_{\{\gamma\in\zz_+^n:\ |\gamma|=s\}}
\frac{\partial_{(1)}^{\gamma}\widetilde{k}_{\beta}(\widetilde{x},y)
-\partial_{(1)}^{\gamma}\widetilde{k}_{\beta}(x_0,y)}
{\gamma!}(x-x_0)^\gamma\r||y|^{|\nu|}\,dy\\
&\quad\ls\lf|\int_{B(x_0,R)}
\widetilde{k}_{\beta}(x,y)y^{\nu}\,dy\r|
+\sum_{\{\gamma\in\zz_+^n:\ |\gamma|\leq s\}}
|x-x_0|^{|\gamma|}\int_{B(x_0,R)\setminus B_0}
\frac{|y|^{|\nu|}}{|y-x_0|^{n+|\gamma|-\beta}}\,dy\\
&\qquad+\int_{\rn\setminus B(x_0,R)}
\frac{|\widetilde{x}-x_0
|^{\dz}|x-x_0|^s|y|^{s}}{|y-x_0|^{n+s+\dz-\beta}}\,dy\\
&\quad\ls\lf|\int_{B(x_0,R)}
\widetilde{k}_{\beta}(x,y)y^{\nu}\,dy\r|
+1,
\end{align*}
where the implicit positive constants depend on $x_0$ and $R$,
and, in the penultimate step, we used \eqref{regular-i}
together with $|y-x_0|\geq R> 2|\widetilde{x}-x_0|$.
Using this, the fact that $y^{\nu}\in L^{q}_{\mathrm{loc}}(\rn)$
for any $q\in(\frac{n}{\beta},\fz)$,
and Lemma \ref{lem2.10x} with $E$ replaced by $B(x_0,R)$,
we conclude that $\widetilde{I}_{\beta,B_0}(y^{\nu})$ is well defined almost
everywhere on $\rn$,
and $\widetilde{I}_{\beta,B_0}(y^{\nu})\mathbf{1}_{\widetilde{B}}\in L^b(\rn)$
with $b\in[1,\fz)$, which hence completes the proof of (i).

Now, we show (ii). Indeed, let
$B_1:=B(x_1,r_1)$ be any given ball of $\rn$
with $x_1\in\rn$ and $r_1\in(0,\fz)$.
Then, by (i) of the present lemma, we conclude that,
for any $\nu\in\zz_+^n$ with $|\nu|\leq s$, and almost every $x\in\rn$,
\begin{align*}
&\widetilde{I}_{\beta,B_0}(y^{\nu})(x)
-\widetilde{I}_{\beta,B_1}(y^{\nu})(x)\\
&\quad=\int_{\rn}\lf[\widetilde{k}_{\beta}(x,y)
-\sum_{\{\gamma\in\zz_+^n:\ |\gamma|\leq s\}}
\frac{\partial_{(1)}^{\gamma}k_{\beta}(x_0,y)}{\gamma!}(x-x_0)^{\gamma}
\mathbf{1}_{\rn\setminus B_0}(y)\r]y^{\nu}\,dy\\
&\qquad-\int_{\rn}\lf[\widetilde{k}_{\beta}(x,y)
-\sum_{\{\gamma\in\zz_+^n:\ |\gamma|\leq s\}}
\frac{\partial_{(1)}^{\gamma}k_{\beta}(x_1,y)}{\gamma!}(x-x_1)^{\gamma}
\mathbf{1}_{\rn\setminus B_1}(y)\r]y^{\nu}\,dy\\
&\quad=\int_{\rn}\lf[\sum_{\{\gamma\in\zz_+^n:\ |\gamma|\leq s\}}
\frac{\partial_{(1)}^{\gamma}
\widetilde{k}_{\beta}(x_1,y)}{\gamma!}(x-x_1)^{\gamma}
\mathbf{1}_{\rn\setminus B_1}(y)\r.\\
&\qquad\lf.-\sum_{\{\gamma\in\zz_+^n:\ |\gamma|\leq s\}}
\frac{\partial_{(1)}^{\gamma}\widetilde{k}_{\beta}
(x_0,y)}{\gamma!}(x-x_0)^{\gamma}
\mathbf{1}_{\rn\setminus B_0}(y)\r]y^{\nu}\,dy\\
&\quad<\fz,
\end{align*}
which, together with Lemma \ref{int-B-P}, further implies that
\begin{align}\label{P-B0-B1'}
\widetilde{I}_{\beta,B_0}(y^{\nu})
-\widetilde{I}_{\beta,B_1}(y^{\nu})\in \mathcal{P}_s(\rn)
\end{align}
after changing values
on a set of measure zero.
This finishes the proof of (ii).

Next, we prove (iii). For any $q\in(1,\fz)$,
$a\in L^q_{s}(\rn)$ such that $\supp\,(a)\subset B:=B(z,r)\subset \rn$
with $z\in\rn$ and $r\in(0,\fz)$, and
$\nu\in\zz_+^n$ with $|\nu|\leq s$, we have
\begin{align}\label{Ia-f'}
&\int_{\rn} I_{\beta}(a)(x)x^{\nu}\,dx\noz\\
&\quad=\int_{\rn}\int_{B}k_{\beta}(x,y)a(y)\,dy\,x^{\nu}\,dx\noz\\
&\quad=\int_{2B}\int_{B}k_{\beta}(x,y)a(y)\,dy\,x^{\nu}\,dx\noz\\
&\qquad+\int_{\rn\setminus 2B}\int_{B}a(y)\lf[k_{\beta}(x,y)
-\sum_{\{\gamma\in\zz_+^n:\ |\gamma|\leq s\}}
\frac{\partial_{(2)}^{\gamma}k_{\beta}(x,z)}{\gamma!}(y-z)^\gamma
\r]\,dy\,x^{\nu}\,dx\noz\\
&\quad=:\mathrm{E}_1+\mathrm{E}_2
\end{align}
and hence we only need to calculate
$\mathrm{E}_1$ and $\mathrm{E}_2$, respectively.

We first consider $\mathrm{E}_1$. Indeed,
let $\mathcal{I}_{\beta}$ be as in \eqref{cla-I}.
If $q\in(1,\frac{n}{\beta})$, then,
using the H\"older inequality and Lemma \ref{fractional}(ii)
with $b:=q$ and $\widetilde{b}:=\frac{qn}{n-q\beta}$,
we find that
\begin{align}\label{3.13x}
\lf\|\mathcal{I}_{\beta}(|a|)\r\|_{L^q(2B_0)}
\ls\lf\|\mathcal{I}_{\beta}(|a|)\r\|_{L^\frac{qn}{n-q\beta}(2B_0)}
\ls\lf\|\mathcal{I}_{\beta}(|a|)\r\|_{L^\frac{qn}{n-q\beta}(\rn)}
\ls\|a\|_{L^q(B_0)}.
\end{align}
If $q\in[\frac{n}{\beta},\fz)$, then choose $q_1\in(1,\frac{n}{\beta})$
such that $q<\frac{nq_1}{n-\beta q_1}$. From this,
the H\"older inequality, and Lemma \ref{fractional}(ii) with $b:=q_1$
and $\widetilde{b}:=\frac{nq_1}{n-\beta q_1}$, we deduce that
\begin{align}\label{2.15x}
\lf\|\mathcal{I}_{\beta}(|a|)\r\|_{L^q(2B_0)}
\ls \lf\|\mathcal{I}_{\beta}(|a|)\r\|_{L^{\frac{nq_1}{n-\beta q_1}}(2B_0)}
\ls \|a\|_{L^{q_1}(B_0)}\ls \|a\|_{L^q(B_0)},
\end{align}
which, combined with \eqref{size-i'} with
$\gamma:=\mathbf{0}$, \eqref{3.13x}, and the fact that
$x^{\nu}\mathbf{1}_{2B}\in L^{q'}(\rn)$,
further implies that
\begin{align*}
\int_{2B}\int_{B}\lf|k_{\beta}(x,y)a(y)\r|\,dy\,\lf|x^{\nu}\r|\,dx
&\ls\int_{2B}\int_{B}\frac{|a(y)|}
{|x-y|^{n-\beta}}\,dy\,\lf|x^{\nu}\r|\,dx\\
&\ls \lf\|\mathcal{I}_{\beta}(|a|)\r\|_{L^q(2B)}
\lf(\int_{2B}\lf|x^{\nu}\r|^{q'}\,dx\r)^{\frac{1}{q'}}<\fz.
\end{align*}
By this and the Fubini theorem, we find that
\begin{align}\label{II-01'}
\mathrm{E}_1=\int_{B}\int_{2B}k_{\beta}(x,y)x^{\nu}\,dx\,a(y)\,dy.
\end{align}
This is a desired conclusion of $\mathrm{E}_1$.

Now, we consider $\mathrm{E}_2$. To this end, we first show that
\begin{align}\label{I-HK-a-1'}
\widetilde{\mathrm{E}}_2
:=&\int_{\rn\setminus 2B}\int_{B}|a(y)|\lf|k_{\beta}(x,y)
-\sum_{\{\gamma\in\zz_+^n:\ |\gamma|\leq s\}}
\frac{\partial_{(2)}^{\gamma}k_{\beta}(x,z)}{\gamma!}
(y-z)^\gamma\r|\noz\\
&\times\lf|x^{\nu}\r|\,dy\,dx<\fz.
\end{align}
Indeed, using the Tonelli theorem, the Taylor remainder theorem,
and $\beta\in(0,\dz)$, we conclude that,
for any $y\in B$, there exists a $\widetilde{y}\in B$ such that,
for any $\nu\in\zz_+^n$ with $|\nu|\leq s$,
\begin{align*}
\widetilde{\mathrm{E}}_2
&=\int_{B}|a(y)|\int_{\rn\setminus 2B}\lf|k_{\beta}(x,y)
-\sum_{\{\gamma\in\zz_+^n:\ |\gamma|\leq s\}}
\frac{\partial_{(2)}^{\gamma}k_{\beta}(x,z)}{\gamma!}
(y-z)^\gamma\r|\lf|x^{\nu}\r|\,dx\,dy\\
&=\int_{B}|a(y)|\int_{\rn\setminus 2B}
\lf|\sum_{\{\gamma\in\zz_+^n:\ |\gamma|=s\}}
\frac{\partial_{(2)}^{\gamma} k_{\beta}(x,\widetilde{y})
-\partial_{(2)}^{\gamma} k_{\beta}(x,z)}{\gamma!}
(y-z)^{\gamma}\r|\lf|x^{\nu}\r|\,dx\,dy\\
&\ls\int_{B}|a(y)|\int_{\rn\setminus 2B}
\frac{|\widetilde{y}-z|^\dz|y-z|^s
|x|^{|\nu|}}{|x-z|^{n+s+\dz-\beta}}\,dx\,dy\\
&\ls r^{s+\dz}\|a\|_{L^1(B)}\int_{\rn\setminus 2B}
\frac{|x|^{|\nu|}}{|x-z|^{n+s+\dz-\beta}}\,dx<\fz,
\end{align*}
where, in the third step, we used \eqref{regular-i'}
together with $|x-z|\geq 2|\widetilde{y}-z|$.
This implies that \eqref{I-HK-a-1'} holds true.
Then, by \eqref{I-HK-a-1'} and the Fubini theorem, we find that
\begin{align}\label{I-02'}
\mathrm{E}_2=\int_{B}a(y)\int_{\rn\setminus 2B}\lf[k_{\beta}(x,y)
-\sum_{\{\gamma\in\zz_+^n:\ |\gamma|\leq s\}}
\frac{\partial_{(2)}^{\gamma}k_{\beta}(x,z)}{\gamma!}(y-z)^\gamma
\r]x^{\nu}\,dx\,dy.
\end{align}
Altogether, from \eqref{Ia-f'}, \eqref{II-01'}, \eqref{I-02'},
\eqref{k-w}, and \eqref{P-B0-B1'},
we deduce that
\begin{align*}
&\int_{\rn} I_{\beta}(a)(x)x^{\nu}\,dx\\
&\quad=\int_{B}a(y)\int_{\rn}\lf[k_{\beta}(x,y)
-\sum_{\{\gamma\in\zz_+^n:\ |\gamma|\leq s\}}
\frac{\partial_{(2)}^{\gamma}k_{\beta}(x,z)}{\gamma!}(y-z)^\gamma
\mathbf1_{\rn\setminus 2B}(x)\r]x^{\nu}\,dx\,dy\noz\\
&\quad=\int_{B}a(y)\int_{\rn}\lf[\widetilde{k}_{\beta}(y,x)
-\sum_{\{\gamma\in\zz_+^n:\ |\gamma|\leq s\}}
\frac{\partial_{(1)}^{\gamma}
\widetilde{k}_{\beta}(z,x)}{\gamma!}(y-z)^\gamma
\mathbf1_{\rn\setminus 2B}(x)\r]x^{\nu}\,dx\,dy\\
&\quad=\int_{B}a(y)\widetilde{I}_{\beta,2B}(x^{\nu})(y)\,dy
=\int_{B}a(y)\widetilde{I}_{\beta,B_0}(x^{\nu})(y)\,dy,
\end{align*}
which implies that \eqref{0-0} holds true.
This finishes the proof of (iii) and hence of Lemma \ref{I-a-x-lem}.
\end{proof}

\begin{pro}\label{I-a-x}
Let $s\in\zz_+$, $\dz\in(0,1]$, $\beta\in(0,\dz)$, and $k_{\beta}$ be an
$s$-order fractional kernel with regularity $\dz$.
Let $\widetilde{k}_{\beta}$ be the adjoint kernel of $k_{\beta}$ as in \eqref{k-w},
$B_0:=B(x_0,r_0)$ a given ball of $\rn$ with $x_0\in\rn$
and $r_0\in(0,\fz)$, and
$\widetilde{I}_{\beta,B_0}$
as in \eqref{D-Iw}. Then \eqref{I-x-gam}
holds true if and only if, for any
$\gamma\in\zz_+^n$ with $|\gamma|\leq s$,
$\widetilde{I}_{\beta,B_0}(y^{\gamma})\in \mathcal{P}_s(\rn)$
after changing values on a set of measure zero.
\end{pro}

\begin{proof}
Let $s$, $\dz$, $\beta$, $k_{\beta}$, $\widetilde{k}_{\beta}$,
$B_0$, and $\widetilde{I}_{\beta,B_0}$
be as in the present proposition.
The sufficiency follows immediately from \eqref{0-0}.	
Next, we show the necessity.
Indeed, using \eqref{pq},
we conclude that, for any $h\in L^{2}(\rn)$
supported in a ball $B\subset \rn$,
$h-P_{B}^{(s)}(h)\in L^{2}_s(B)$. By this,
\eqref{I-x-gam} with $a$ replaced by
$[h-P_{B}^{(s)}(h)]\mathbf{1}_B$, \eqref{0-0} with $q:=2$
and $a$ replaced by $[h-P_{B}^{(s)}(h)]\mathbf{1}_B$,
and \eqref{pq}, we conclude that, for any $h\in L^{2}(\rn)$,
\begin{align}\label{arb-h}
0&=\int_{\rn} I_{\beta}\lf(\lf[h-P_{B}^{(s)}(h)\r]
\mathbf{1}_{B}\r)(x)x^{\gamma}\,dx\noz\\
&=\int_{B}\lf[h(x)-P_{B}^{(s)}(h)(x)\r]
\widetilde{I}_{\beta,B_0}(y^{\gamma})(x)\,dx\noz\\
&=\int_{B}\lf[h(x)-P_{B}^{(s)}(h)(x)\r]
\lf[\widetilde{I}_{\beta,B_0}(y^{\gamma})(x)
-P_{B}^{(s)}\lf(\widetilde{I}_{\beta,B_0}(y^{\gamma})\r)(x)\r]\,dx\noz\\
&=\int_{B}h(x)\lf[\widetilde{I}_{\beta,B_0}(y^{\gamma})(x)
-P_{B}^{(s)}\lf(\widetilde{I}_{\beta,B_0}(y^{\gamma})\r)(x)\r]\,dx,
\end{align}
where $I_\beta$ is the fractional integral with kernel $k_{\beta}$.
Moreover, from Lemma \ref{I-a-x-lem}(i), we deduce that,
for any $\gamma\in\zz_+^n$ with $|\gamma|\leq s$,
$$\widetilde{I}_{\beta,B_0}(y^{\gamma})
-P_{B}^{(s)}(\widetilde{I}_{\beta,B_0}(y^{\gamma}))\in L^2(B).$$
By this and \eqref{arb-h}, we conclude that,
for almost every $x\in B$,
\begin{align*}
\widetilde{I}_{\beta,B_0}(y^{\gamma})(x)
=P_{B}^{(s)}\lf(\widetilde{I}_{\beta,B_0}(y^{\gamma})\r)(x).
\end{align*}
Repeating the above procedure with $B$ replaced by $2^kB$
for any given $k\in\nn$,
we find that, for almost every $x\in 2^kB$,
\begin{align*}
\widetilde{I}_{\beta,B_0}(y^{\gamma})(x)
=P_{2^kB}^{(s)}\lf(\widetilde{I}_{\beta,B_0}(y^{\gamma})\r)(x)
\end{align*}
and hence, for almost every $x\in\rn$,
\begin{align*}
\widetilde{I}_{\beta,B_0}(y^{\gamma})(x)
=P_{B}^{(s)}\lf(\widetilde{I}_{\beta,B_0}(y^{\gamma})\r)(x).
\end{align*}
This finishes the proof of the necessity
and hence of Proposition \ref{I-a-x}.
\end{proof}

\begin{cor}\label{coro2.16x}
Let $s\in\zz_+$, $\dz\in(0,1]$, $\beta\in(0,\dz)$,
$k_{\beta}$ be an $s$-order fractional kernel with regularity $\dz$
as in Definition \ref{I-k}, and
$I_{\beta}$ the fractional integral with kernel $k_{\beta}$. Then
\eqref{I-x-gam} holds true if and only if, for any $q\in(1,\fz)$,
any $a\in L_s^q(\rn)$ having bounded support,
and any $\gamma\in\zz_+^n$ with $|\gamma|\leq s$, it holds true that
\begin{align*}
\int_{\rn} I_{\beta}(a)(x)x^\gamma\,dx=0.
\end{align*}
\end{cor}

\begin{proof}
Let $s$, $\dz$, $\beta$, $k_{\beta}$, and $I_{\beta}$ be as in the present
corollary. The sufficiency obviously holds true by
taking $q=2$.
Now, we show the necessity.
Using \eqref{I-x-gam} and Proposition \ref{I-a-x}, we have, for any
$\gamma\in\zz_+^n$ with $|\gamma|\leq s$,
$\widetilde{I}_{\beta,B_0}(y^{\gamma})\in \mathcal{P}_s(\rn)$
after changing values on a set of measure zero.
From this and \eqref{0-0}, we deduce that, for any $q\in(1,\fz)$,
any $a\in L_s^q(\rn)$ having bounded support,
and any $\gamma\in\zz_+^n$ with $|\gamma|\leq s$,
$$
\int_{\rn} I_{\beta}(a)(x)x^\gamma\,dx
=\int_{\rn}a(x)\widetilde{I}_{\beta,B_0}(y^{\gamma})(x)\,dx=0,
$$
where $\widetilde{I}_{\beta,B_0}$ is as in \eqref{D-Iw}.
This finishes the proof of the necessity
and hence of Corollary \ref{coro2.16x}.
\end{proof}

The following lemma is just \cite[Proposition 2.2]{jtyyz1}
which gives an equivalent characterization of
$JN_{(p,q,s)_\alpha}^{\mathrm{con}}(\rn)$.

\begin{lem}\label{AC}
Let $p$, $q\in[1,\infty)$, $s\in\zz_+$, and $\alpha\in \rr$.
Then $f\in JN_{(p,q,s)_\alpha}^{\mathrm{con}}(\rn)$ if and only if
$f\in L^q_{\mathrm{loc}}(\rn)$ and
\begin{align*}
&[f]_{JN_{(p,q,s)\alpha}^{\mathrm{con}}(\rn)}\\
&\quad:=\sup_{r\in(0,\infty)}\lf[\int_{\rn}\lf\{|B(y,r)|^{-\alpha}
\lf[\fint_{B(y,r)}\lf|f(x)-P_{B(y,r)}^{(s)}(f)(x)\r|^q\,dx
\r]^{\frac{1}{q}}\r\}^p\,dy\r]^{\frac{1}{p}}<\fz.
\end{align*}
Moreover, for any $f\in JN_{(p,q,s)_\alpha}^{\mathrm{con}}(\rn)$,
$$
\|f\|_{JN_{(p,q,s)\alpha}^{\mathrm{con}}(\rn)}
\sim [f]_{JN_{(p,q,s)\alpha}^{\mathrm{con}}(\rn)},
$$
where the positive equivalence constants are independent of $f$.
\end{lem}

Next, we show that $\widetilde{I}_{\beta,B_0}$ in
Definition \ref{I-w} is well defined on
$JN_{(p,q,s)_\alpha}^{\mathrm{con}}(\rn)$
via borrowing some ideas from \cite[Theorem 4.1]{N10}.

\begin{pro}\label{I-w-h}
Let $p\in[1,\fz]$, $q\in[1,\fz)$, $s\in\zz_+$, $\dz\in(0,1]$,
$\beta\in(0,\dz)$, $\alpha\in (-\fz,\frac{1}{p}+\frac{s+\dz-\beta}{n})$,
and $k_{\beta}$ be an $s$-order fractional kernel with regularity $\dz$
such that the fractional integral $I_\beta$, with kernel $k_{\beta}$,
has the vanishing moments up to order $s$.
Let $\widetilde{k}_{\beta}$ be the adjoint kernel of $k_{\beta}$ as in \eqref{k-w},
$B_0:=B(x_0,r_0)$ a given ball of $\rn$
with $x_0\in\rn$ and $r_0\in (0,\fz)$, and
$\widetilde{I}_{\beta,B_0}$ as in Definition \ref{I-w}
with kernel $\widetilde{k}_{\beta}$.
Then, for any $f\in JN_{(p,q,s)_\alpha}^{\mathrm{con}}(\rn)$,
$\widetilde{I}_{\beta,B_0}(f)$
is well defined almost everywhere on $\rn$.
\end{pro}

\begin{proof}
Let $p$, $q$, $s$, $\alpha$, $\dz$, $\beta$, $k_{\beta}$,
$I_{\beta}$, $\widetilde{k}_{\beta}$, $B_0:=B(x_0,r_0)$
with $x_0\in\rn$ and $r_0\in (0,\fz)$, and $\widetilde{I}_{\beta,B_0}$
be as in the present proposition.
We only prove the case $p\in[1,\fz)$ because the proof of the case
$p=\fz$ is similar. Let $B(z,r)$ be any given ball of $\rn$
with $z\in\rn$ and $r\in(0,\fz)$.
Then, by the definition of $\widetilde{I}_{\beta,B_0}$,
we conclude that, for any given $f\in JN_{(p,q,s)_\alpha}^{\mathrm{con}}(\rn)$
and any $x\in B(z,r)$,
\begin{align}\label{I-1}
\widetilde{I}_{\beta,B_0}(f)(x)
=F_{B(z,r)}^{(1)}(x)+F_{B(z,r)}^{(2)}(x)
+P_{B(z,r),B_0}^{(1)}(x)+P_{B(z,r),B_0}^{(2)}(x),
\end{align}
where
\begin{align}\label{2.22x}
F_{B(z,r)}^{(1)}(x):=\int_{2B(z,r)}\widetilde{k}_{\beta}(x,y)
\lf[f(y)-P_{2B(z,r)}^{(s)}(f)(y)\r]\,dy,
\end{align}

\begin{align}\label{2.22y}
F_{B(z,r)}^{(2)}(x):=&\int_{\rn\setminus 2B(z,r)}
\lf[\widetilde{k}_{\beta}(x,y)
-\sum_{\{\gamma\in\zz_+^n:\ |\gamma|\leq s\}}
\frac{\partial_{(1)}^{\gamma}
\widetilde{k}_{\beta}(z,y)}{\gamma!}(x-z)^{\gamma}\r]\noz\\
&\times\lf[f(y)-P_{2B(z,r)}^{(s)}(f)(y)\r]\,dy,
\end{align}

\begin{align}\label{2.22z}
P_{B(z,r),B_0}^{(1)}(x):=\int_{\rn}
K_{B(z,r),B_0}(x,y)\lf[f(y)-P_{2B(z,r)}^{(s)}(f)(y)\r]\,dy
\end{align}
with
\begin{align*}
K_{B(z,r),B_0}(x,y)
&:=\sum_{\{\gamma\in\zz_+^n:\ |\gamma|\leq s\}}
\frac{\partial_{(1)}^{\gamma}\widetilde{k}_{\beta}(z,y)}{\gamma!}(x-z)^{\gamma}
\mathbf1_{\rn\setminus 2B(z,r)}(y)\\
&\quad-\sum_{\{\gamma\in\zz_+^n:\ |\gamma|\leq s\}}
\frac{\partial_{(1)}^{\gamma}
\widetilde{k}_{\beta}(x_0,y)}{\gamma!}(x-x_0)^{\gamma}
\mathbf1_{\rn\setminus B_0}(y),
\end{align*}
and
\begin{align}\label{2.22w}
P_{B(z,r),B_0}^{(2)}(x):=&\int_{\rn}
\lf[\widetilde{k}_{\beta}(x,y)
-\sum_{\{\gamma\in\zz_+^n:\ |\gamma|\leq s\}}
\frac{\partial_{(1)}^{\gamma}
\widetilde{k}_{\beta}(x_0,y)}{\gamma!}(x-x_0)^{\gamma}
\mathbf1_{\rn\setminus B_0}(y)\r]\noz\\
&\times P_{2B(z,r)}^{(s)}(f)(y)\,dy.
\end{align}

We first consider $F_{B(z,r)}^{(1)}$.
Indeed, since $f\in JN_{(p,q,s)_\alpha}^{\mathrm{con}}(\rn)$,
it follows that
\begin{align*}
|2B(z,r)|^{\frac1p-\az}\lf[\fint_{2B(z,r)}
\lf|f(y)-P_{2B(z,r)}^{(s)}(f)(y)\r|^q\,dy\r]^\frac1q
\ls\|f\|_{JN_{(p,q,s)_\alpha}^{\mathrm{con}}(\rn)}<\fz.
\end{align*}
This implies that
$[f-P_{2B(z,r)}^{(s)}(f)]\mathbf1_{2B(z,r)}\in L^q(\rn)$.
From this and Lemma \ref{lem2.10x},
we deduce that $F_{B(z,r)}^{(1)}(x)$ is well defined
for almost every $x\in B(z,r)$.

Now, we consider $F_{B(z,r)}^{(2)}$.
By the Taylor remainder theorem,
we find that, for any $x\in B(z,r)$,
there exists an $\widetilde{x}\in B(z,r)$ such that, for any $x\in B(z,r)$,
\begin{align}\label{I-E-1}
\lf|F_{B(z,r)}^{(2)}(x)\r|
&\leq\int_{\rn\setminus 2B(z,r)}
\lf|\widetilde{k}_{\beta}(x,y)
-\sum_{\{\gamma\in\zz_+^n:\ |\gamma|\leq s\}}
\frac{\partial_{(1)}^{\gamma}
\widetilde{k}_{\beta}(z,y)}{\gamma!}(x-z)^{\gamma}\r|\noz\\
&\quad\times\lf|f(y)-P_{2B(z,r)}^{(s)}(f)(y)\r|\,dy\noz\\
&=\int_{\rn\setminus 2B(z,r)}
\lf|\sum_{\{\gamma\in\zz_+^n:\ |\gamma|=s\}}
\frac{\partial_{(1)}^{\gamma} \widetilde{k}_{\beta}(\widetilde{x},y)
-\partial_{(1)}^{\gamma} \widetilde{k}_{\beta}(z,y)}
{\gamma!}(x-z)^{\gamma}\r|\noz\\
&\quad\times\lf|f(y)-P_{2B(z,r)}^{(s)}(f)(y)\r|\,dy\noz\\
&\lesssim \int_{\rn\setminus 2B(z,r)}
\frac{|\widetilde{x}-x|^{\dz}|x-z|^{s}}{|y-z|^{n+s+\dz-\beta}}
\lf|f(y)-P_{2B(z,r)}^{(s)}(f)(y)\r|\,dy\noz\\
&\lesssim r^{s+\dz}\int_{\rn\setminus 2B(z,r)}
\frac{|f(y)-P_{2B(z,r)}^{(s)}(f)(y)|}{|y-z|^{n+s+\dz-\beta}}\,dy,
\end{align}
where, in the penultimate step, we used \eqref{regular-i}
together with $|y-z|\geq2|\widetilde{x}-z|$
for any $y\in \rn\setminus 2B(z,r)$.
Using this, $\alpha\in (-\fz,\frac{1}{p}+\frac{s+\dz-\beta}{n})$,
and Lemma \ref{I-JN} with
$\lz:=s+\dz-\beta\in(s,\fz)$, we conclude that
\begin{align*}
\lf|F_{B(z,r)}^{(2)}(x)\r|
\lesssim\|f\|_{JN_{(p,q,s)_\alpha}^{\mathrm{con}}(\rn)}<\fz.
\end{align*}
This shows that $F_{B(z,r)}^{(2)}(x)$
is well defined for any $x\in B(z,r)$.

Next, we consider $P_{B(z,r),B_0}^{(1)}$.
To this end, let $R:=3|z-x_0|+2r+r_0$.
Then $[B_0\cup 2B(z,r)]\subset B(z,R)$. From this
and \eqref{size-i}, we deduce that,
for any $x\in B(z,r)$ and $y\in B(z,R)$,
\begin{align}\label{2.26x}
&\lf|K_{B(z,r),B_0}(x,y)\r|\noz\\
&\quad\leq\lf|\sum_{\{\gamma\in\zz_+^n:\ |\gamma|\leq s\}}
\frac{\partial_{(1)}^{\gamma}
\widetilde{k}_{\beta}(z,y)}{\gamma!}(x-z)^{\gamma}
\mathbf1_{\rn\setminus 2B(z,r)}(y)\r|\noz\\
&\qquad+\lf|\sum_{\{\gamma\in\zz_+^n:\ |\gamma|\leq s\}}
\frac{\partial_{(1)}^{\gamma}
\widetilde{k}_{\beta}(x_0,y)}{\gamma!}(x-x_0)^{\gamma}
\mathbf1_{\rn\setminus B_0}(y)\r|\noz\\
&\quad\ls\sum_{\{\gamma\in\zz_+^n:\ |\gamma|\leq s\}}
\frac{|x-z|^{|\gamma|}}{|y-z|^{n+|\gamma|}}
\mathbf1_{\rn\setminus 2B(z,r)}(y)
+\sum_{\{\gamma\in\zz_+^n:\ |\gamma|\leq s\}}
\frac{|x-x_0|^{\gamma}}{|y-x_0|^{n+|\gamma|}}
\mathbf1_{\rn\setminus B_0}(y)\noz\\
&\quad\ls\sum_{\{\gamma\in\zz_+^n:\ |\gamma|\leq s\}}
\frac{|x-z|^{|\gamma|}}{r^{n+|\gamma|-\beta}}
+\sum_{\{\gamma\in\zz_+^n:\ |\gamma|\leq s\}}
\frac{|x-x_0|^{|\gamma|}}{r_0^{n+|\gamma|-\beta}}\noz\\
&\quad\ls\frac{1}{r^{n-\beta}}
+\sum_{\{\gamma\in\zz_+^n:\ |\gamma|\leq s\}}
\frac{(r+|z-x_0|)^{\gamma}}{r_0^{n+|\gamma|-\beta}}
\sim 1,
\end{align}
where the implicit positive constants depend on $x_0$, $r_0$, $z$, and $r$.
Moreover, using the Taylor remainder
theorem, we find that,
for any $x\in B(z,r)$ and $y\in \rn\setminus B(z,R)$,
there exist a $z_1\in B(z,r)$ and a
$z_2\in B(x_0,|z-x_0|+r)\subset B(z,R)$ such that
\begin{align*}
\lf|K_{B(z,r),B_0}(x,y)\r|
&\leq\lf|\sum_{\{\gamma\in\zz_+^n:\ |\gamma|\leq s\}}
\frac{\partial_{(1)}^{\gamma}
\widetilde{k}_{\beta}(z,y)}{\gamma!}(x-z)^{\gamma}
-\widetilde{k}_{\beta}(x,y)\r|\\
&\quad+\lf|\widetilde{k}_{\beta}(x,y)
-\sum_{\{\gamma\in\zz_+^n:\ |\gamma|\leq s\}}
\frac{\partial_{(1)}^{\gamma}
\widetilde{k}_{\beta}(x_0,y)}{\gamma!}(x-x_0)^{\gamma}\r|\\
&=\lf|\sum_{\{\gamma\in\zz_+^n:\ |\gamma|=s\}}
\frac{\partial_{(1)}^{\gamma}\widetilde{k}_{\beta}(z,y)
-\partial_{(1)}^{\gamma}\widetilde{k}_{\beta}(z_1,y)}{\gamma!}
(x-z)^{\gamma}\r|\\
&\quad+\lf|\sum_{\{\gamma\in\zz_+^n:\ |\gamma|=s\}}
\frac{\partial_{(1)}^{\gamma}\widetilde{k}_{\beta}(z_2,y)
-\partial_{(1)}^{\gamma}\widetilde{k}_{\beta}(x_0,y)}{\gamma!}
(x-x_0)^{\gamma}\r|\\
&\ls\sum_{\{\gamma\in\zz_+^n:\ |\gamma|=s\}}
\frac{|z-z_1|^{\delta}|x-z|^{|\gamma|}}{|y-z|^{n+|\gamma|+\dz-\beta}}
+\sum_{\{\gamma\in\zz_+^n:\ |\gamma|=s\}}
\frac{|z_2-x_0|^{\delta}|x-x_0|^{|\gamma|}}{|y-z|^{n+|\gamma|+\dz-\beta}}\\
&\ls\frac{r^{s+\delta}+(r+|z-x_0|)^{s+\dz}}{|y-z|^{n+s+\dz-\beta}}
\sim\frac{1}{|y-z|^{n+s+\dz-\beta}},
\end{align*}
where the implicit positive constants depend on $x_0$, $r_0$, $z$, and $r$,
and, in the third step, we used \eqref{regular-i}
together with $|y-z|\geq 2|z_1-z|$, $|y-x_0|\geq 2|z_2-x_0|$,
and $|y-x_0|\sim |y-z|$.
Combining this with \eqref{2.26x}, the H\"older inequality,
$\alpha\in (-\fz,\frac{1}{p}+\frac{s+\dz-\beta}{n})$,
and Lemma \ref{I-JN} with $\lz:=s+\dz-\beta\in(s,\fz)$,
we conclude that, for any $x\in B(z,r)$,
\begin{align*}
\lf|P_{B(z,r),B_0}^{(1)}(x)\r|
&\leq\int_{\rn}\lf|K_{B(z,r),B_0}(x,y)\r|
\lf|f(y)-P_{2B(z,r)}^{(s)}(f)(y)\r|\,dy\\
&\lesssim\int_{B(z,R)}\lf|f(y)-P_{2B(z,r)}^{(s)}(f)(y)\r|\,dy\\
&\quad+\int_{\rn\setminus 2B(z,r)}\frac{|f(y)-P_{2B(z,r)}^{(s)}(f)(y)|}
{|z-y|^{n+s+\dz-\beta}}\,dy\\
&\lesssim\lf\|f-P_{2B(z,r)}^{(s)}(f)\r\|_{L^1(B(z,R))}
+\|f\|_{JN_{(p,q,s)_\alpha}^{\mathrm{con}}(\rn)}
<\fz.
\end{align*}
This implies that $P_{B(z,r),B_0}^{(1)}(x)$
is well defined for any $x\in B(z,r)$.
Furthermore, by Lemma \ref{int-B-P}, we find that
\begin{align}\label{2.27x}
P_{B(z,r),B_0}^{(1)}\in \mathcal{P}_{s}(B(z,r)).
\end{align}

Finally, Proposition \ref{I-a-x} implies that $P_{B(z,r),B_0}^{(2)}(x)$
is well defined for almost every $x\in B(z,r)$.
This, together with the arbitrariness
of $B(z,r)$, then finishes the proof of Proposition \ref{I-w-h}.
\end{proof}

\begin{rem}\label{Rem-Iw}
In Definition \ref{I-w}, we claim that,
for any given ball $B_1$ of $\rn$,
$$\widetilde{I}_{\beta,B_0}(f)
-\widetilde{I}_{\beta,B_1}(f)\in \mathcal{P}_s(\rn)$$
after changing values on a set of measure zero.
Indeed, from \eqref{D-Iw} and Proposition \ref{I-w-h}, we deduce that,
for any given ball $B_1:=B(x_1,r_1)$ with $x_1\in\rn$ and $r_1\in(0,\fz)$,
and for almost every $x\in\rn$,
\begin{align*}
&\widetilde{I}_{\beta,B_0}(f)(x)
-\widetilde{I}_{\beta,B_1}(f)(x)\\
&\quad=\int_{\rn}\lf[\widetilde{k}_{\beta}(x,y)
-\sum_{\{\gamma\in\zz_+^n:\ |\gamma|\leq s\}}
\frac{\partial_{(1)}^{\gamma}
\widetilde{k}_{\beta}(x_0,y)}{\gamma!}(x-x_0)^{\gamma}
\mathbf{1}_{\rn\setminus B_0}(y)\r]f(y)\,dy\\
&\qquad-\int_{\rn}\lf[\widetilde{k}_{\beta}(x,y)
-\sum_{\{\gamma\in\zz_+^n:\ |\gamma|\leq s\}}
\frac{\partial_{(1)}^{\gamma}
\widetilde{k}_{\beta}(x_1,y)}{\gamma!}(x-x_1)^{\gamma}
\mathbf{1}_{\rn\setminus B_1}(y)\r]f(y)\,dy\\
&\quad=\int_{\rn}\lf[\sum_{\{\gamma\in\zz_+^n:\ |\gamma|\leq s\}}
\frac{\partial_{(1)}^{\gamma}
\widetilde{k}_{\beta}(x_1,y)}{\gamma!}(x-x_1)^{\gamma}
\mathbf{1}_{\rn\setminus B_1}(y)\r.\\
&\qquad\lf.-\sum_{\{\gamma\in\zz_+^n:\ |\gamma|\leq s\}}
\frac{\partial_{(1)}^{\gamma}
\widetilde{k}_{\beta}(x_0,y)}{\gamma!}(x-x_0)^{\gamma}
\mathbf{1}_{\rn\setminus B_0}(y)\r]f(y)\,dy\\
&\quad<\fz,
\end{align*}
which, combined with Lemma \ref{int-B-P}, further implies that
$\widetilde{I}_{\beta,B_0}(f)
-\widetilde{I}_{\beta,B_1}(f)\in \mathcal{P}_s(\rn)$
after changing values on a set of measure zero.
Based on this remark and Proposition \ref{prop-2.18x},
in what follows, we write $\widetilde{I}_{\beta}$
instead of $\widetilde{I}_{\beta,B_0}$ if there exists no confusion.
\end{rem}

Now, we establish the boundedness of $\widetilde{I}_\beta$
from $JN_{(p,q_1,s)_\alpha}^{\mathrm{con}}(\rn)$
to $JN_{(p,q_2,s)_{\alpha+\beta/n}}^{\mathrm{con}}(\rn)$
via borrowing some ideas from \cite[Section 5]{N10}.

\begin{thm}\label{Iw-JN-bound}
Let $p\in[1,\fz]$, $q_1$, $q_2\in[1,\fz)$, $s\in\zz_+$, $\dz\in(0,1]$,
$\beta\in (0,\dz)$, $\alpha\in (-\fz,\frac{s+\dz-\beta}{n})$,
$k_{\beta}$ be an $s$-order fractional kernel with regularity $\dz$,
and $I_\beta$ the fractional integral with kernel $k_{\beta}$.
Let $\widetilde{k}_{\beta}$ be the adjoint kernel of $k_{\beta}$ as in \eqref{k-w},
and $\widetilde{I}_{\beta}$ as in Definition \ref{I-w} and Remark \ref{Rem-Iw}
with kernel $\widetilde{k}_{\beta}$.
Then the following two statements are equivalent:
\begin{enumerate}
\item[\rm(i)]
if $q_1=1$ and $q_2\in[1,\frac{n}{n-\beta})$,
or $q_1\in(1,\frac{n}{\beta})$
and $q_2\in[1,\frac{nq_1}{n-\beta q_1}]$,
or $q_1\in[\frac{n}{\beta},\fz)$
and $q_2\in [1,\fz)$, then $\widetilde{I}_{\beta}$
is bounded from $JN_{(p,q_1,s)_\alpha}^{\mathrm{con}}(\rn)$
to $JN_{(p,q_2,s)_{\alpha+\beta/n}}^{\mathrm{con}}(\rn)$, namely,
there exists a positive constant $C$ such that,
for any $f\in JN_{(p,q_1,s)_\alpha}^{\mathrm{con}}(\rn)$,
\begin{align*}
\lf\|\widetilde{I}_{\beta}(f)
\r\|_{JN_{(p,q_2,s)_{\alpha+\beta/n}}^{\mathrm{con}}(\rn)}
\leq C\|f\|_{JN_{(p,q_1,s)_\alpha}^{\mathrm{con}}(\rn)};
\end{align*}
\item[\rm(ii)]
$I_{\beta}$ has the vanishing moments up to order $s$
as in Definition \ref{Def-I-s}.
\end{enumerate}
\end{thm}

\begin{proof}
Let $p$, $q_1$, $q_2$, $s$, $\alpha$, $\dz$,
$\beta$, $k_{\beta}$, $\widetilde{k}_{\beta}$, and
$\widetilde{I}_{\beta}$ be as in the present theorem.
We only prove the case $p\in[1,\fz)$
because the proof of $p=\fz$ is similar.

We first show that (i) $\Rightarrow$ (ii).
Indeed, if $\widetilde{I}_{\beta}$
is bounded from $JN_{(p,q_1,s)_\alpha}^{\mathrm{con}}(\rn)$
to $JN_{(p,q_2,s)_{\alpha+\beta/n}}^{\mathrm{con}}(\rn)$,
then, for any $\gamma\in\zz_+^n$ with $|\gamma|\leq s$,
$$
\lf\|\widetilde{I}_{\beta}(x^{\gamma})
\r\|_{JN_{(p,q_2,s)_{\alpha+\beta/n}}^{\mathrm{con}}(\rn)}
\ls \lf\|x^{\gamma}\r\|_{JN_{(p,q_1,s)_\alpha}^{\mathrm{con}}(\rn)}=0
$$
and hence $\widetilde{I}_{\beta}(x^{\gamma})\in \mathcal{P}_s(\rn)$
after changing values on a set of measure zero.
Using this and Proposition \ref{I-a-x}, we find that
$I_{\beta}$ has the vanishing moments up to order $s$.
This finishes the proof of that (i) $\Rightarrow$ (ii).

Next, we show that (ii) $\Rightarrow$ (i).
Observe that, for any $g\in L^1_{\mathrm{loc}}(\rn)$
and any ball $B$ of $\rn$,
\begin{align*}
\lf[\fint_B\lf|g(x)-P_B^{(s)}(g)(x)
\r|^{q_2}\,dx\r]^{\frac{1}{q_2}}
\sim\inf_{P\in \mathcal{P}_s(B)}
\lf[\fint_B|g(x)-P(x)|^{q_2}\,dx\r]^{\frac{1}{q_2}};
\end{align*}
see, for instance, \cite[(2.12)]{jtyyz1}.
From this, \eqref{I-1}, \eqref{2.27x}, Proposition \ref{I-a-x},
$P_{B}^{(s)}(P)=P$ for any $P\in \mathcal{P}_s(\rn)$
and any ball $B\subset \rn$, and
the Minkowski inequality, we deduce that, for any $r\in(0,\fz)$
and $f\in JN_{(p,q_1,s)_\alpha}^{\mathrm{con}}(\rn)$,
\begin{align}\label{Iw-JN-01}
&\lf[\int_{\rn}
\lf\{|B(z,r)|^{-\alpha-\frac{\beta}{n}}\lf[\fint_{B(z,r)}
\lf|\widetilde{I}_{\beta,B_0}(f)(x)
-P_{B(z,r)}^{(s)}(\widetilde{I}_{\beta,B_0}(f))(x)
\r|^{q_2}\,dx\r]^{\frac{1}{q_2}}\r\}^{p}\,dz\r]^\frac{1}{p}\noz\\
&\quad\lesssim\lf[\int_{\rn}
\lf\{|B(z,r)|^{-\alpha-\frac{\beta}{n}}\lf[\fint_{B(z,r)}
\lf|\widetilde{I}_\beta f(x)
-P_{B(z,r),B_0}^{(1)}(x)-P_{B(z,r),B_0}^{(2)}(x)\r|^{q_2}\,dx
\r]^{\frac{1}{q_2}}\r\}^{p}\,dz\r]^\frac{1}{p}\noz\\
&\quad\sim\lf[\int_{\rn}
\lf\{|B(z,r)|^{-\alpha-\frac{\beta}{n}}\lf[\fint_{B(z,r)}
\lf|F_{B(z,r)}^{(1)}(x)+F_{B(z,r)}^{(2)}(x)\r|^{q_2}\,dx
\r]^{\frac{1}{q_2}}\r\}^{p}\,dz\r]^\frac{1}{p}\noz\\
&\quad\ls\mathrm{Z}_1+\mathrm{Z}_2,
\end{align}
where $F_{B(z,r)}^{(1)}$, $F_{B(z,r)}^{(2)}$, $P_{B(z,r),B_0}^{(1)}$,
and $P_{B(z,r),B_0}^{(2)}$ are, respectively, as in \eqref{2.22x},
\eqref{2.22y}, \eqref{2.22z}, and \eqref{2.22w},
$$
\mathrm{Z}_1:=\lf[\int_{\rn}
\lf\{|B(z,r)|^{-\alpha-\frac{\beta}{n}}\lf[\fint_{B(z,r)}
\lf|F_{B(z,r)}^{(1)}(x)\r|^{q_2}\,dx
\r]^{\frac{1}{q_2}}\r\}^{p}\,dz\r]^\frac{1}{p},
$$
and
$$
\mathrm{Z}_2:=\lf[\int_{\rn}
\lf\{|B(z,r)|^{-\alpha-\frac{\beta}{n}}\lf[\fint_{B(z,r)}
\lf|F_{B(z,r)}^{(2)}(x)\r|^{q_2}\,dx
\r]^{\frac{1}{q_2}}\r\}^{p}\,dz\r]^\frac{1}{p}.
$$

Next, we show that
\begin{align}\label{A1}
\sup_{r\in(0,\fz)}\mathrm{Z}_1
\lesssim \|f\|_{JN_{(p,q_1,s)_\alpha}^{\mathrm{con}}(\rn)}.
\end{align}
To this end, we first prove that, for any given
ball $B(z,r)\subset \rn$ with $z\in\rn$ and $r\in(0,\fz)$,
\begin{align}\label{A1-1}
&\lf|B(z,r)\r|^{-\alpha-\frac{\beta}{n}}
\lf[\fint_{B(z,r)}\lf|F_{B(z,r)}^{(1)}(x)
\r|^{q_2}\,dx\r]^{\frac{1}{q_2}}\noz\\
&\quad\lesssim\lf|2B(z,r)\r|^{-\alpha}
\lf[\fint_{2B(z,r)}\lf|f(x)-P_{2B(z,r)}^{(s)}(f)(x)
\r|^{q_1}\,dx\r]^{\frac{1}{q_1}}
\end{align}
via considering the following three cases on $q_1$ and $q_2$.

\emph{Case i)} $q_1=1$ and $q_2\in[1,\frac{n}{n-\beta})$.
In this case, we have $-(n-\beta)q_2+n-1>-1$. By this,
\eqref{2.22x}, \eqref{size-i}, the Minkowski inequality,
the fact $B(z,r)\subset B(y,4r)$ for any $y\in 2B(z,r)$,
and the H\"older inequality,
we conclude that
\begin{align*}
&\lf|B(z,r)\r|^{-\alpha-\frac{\beta}{n}}
\lf[\fint_{B(z,r)}\lf|F_{B(z,r)}^{(1)}(x)
\r|^{q_2}\,dx\r]^{\frac{1}{q_2}}\\
&\quad\ls\lf|B(z,r)\r|^{-\alpha-\frac{\beta}{n}}
\lf\{\fint_{B(z,r)}\lf[
\int_{2B(z,r)}\frac{|f(y)-P_{2B(z,r)}^{(s)}(f)(y)|}{|x-y|^{n-\beta}}\,dy
\r]^{q_2}\,dx\r\}^{\frac{1}{q_2}}\\
&\quad\lesssim\lf|B(z,r)\r|^{-\alpha-\frac{\beta}{n}-\frac{1}{q_2}}
\int_{2B(z,r)}\lf|f(y)-P_{2B(z,r)}^{(s)}(f)(y)\r|
\lf[\int_{B(z,r)}\frac{1}{|x-y|^{(n-\beta)q_2}}
\,dx\r]^{\frac{1}{q_2}}\,dy\\
&\quad\lesssim\lf|B(z,r)\r|^{-\alpha-\frac{\beta}{n}-\frac{1}{q_2}}
\int_{2B(z,r)}\lf|f(y)-P_{2B(z,r)}^{(s)}(f)(y)\r|
\lf[\int_{B(y,4r)}\frac{1}{|x-y|^{(n-\beta)q_2}}
\,dx\r]^{\frac{1}{q_2}}\,dy\\
&\quad\sim\lf|B(z,r)\r|^{-\alpha-\frac{\beta}{n}-\frac{1}{q_2}}
\int_{2B(z,r)}\lf|f(y)-P_{2B(z,r)}^{(s)}(f)(y)\r|
\lf[\int_0^{4r}t^{-(n-\beta)q_2+n-1}\,dt\r]^{\frac{1}{q_2}}\,dy\\
&\quad\sim\lf|2B(z,r)\r|^{-\alpha}
\fint_{2B(z,r)}\lf|f(y)-P_{2B(z,r)}^{(s)}(f)(y)\r|\,dy\\
&\quad\ls\lf|2B(z,r)\r|^{-\alpha}
\lf[\fint_{2B(z,r)}\lf|f(y)-P_{2B(z,r)}^{(s)}(f)(y)
\r|^{q_1}\,dy\r]^{\frac{1}{q_1}},
\end{align*}
which shows that \eqref{A1-1} holds true in this case.

\emph{Case ii)} $q_1\in(1,\frac{n}{\beta})$
and $q_2\in[1,\frac{nq_1}{n-\beta q_1}]$.
In this case, using the H\"older inequality,
and Lemma \ref{fractional}(ii) with $b:=q_1$,
$\widetilde{b}:=\frac{nq_1}{n-\beta q_1}$, and $g$ replaced by
$[f-P_{2B(z,r)}^{(s)}(f)]\mathbf{1}_{2B(z,r)}$, we have
\begin{align*}
&\lf|B(z,r)\r|^{-\alpha-\frac{\beta}{n}}
\lf[\fint_{B(z,r)}\lf|F_{B(z,r)}^{(1)}(x)
\r|^{q_2}\,dx\r]^{\frac{1}{q_2}}\\
&\quad\leq\lf|B(z,r)\r|^{-\alpha-\frac{\beta}{n}}
\lf[\fint_{B(z,r)}\lf|F_{B(z,r)}^{(1)}(x)
\r|^{\frac{nq_1}{n-\beta q_1}}\,dx\r]^{\frac{n-\beta q_1}{nq_1}}\\
&\quad=\lf|B(z,r)\r|^{-\alpha-\frac{\beta}{n}}
\lf\{\fint_{B(z,r)}\lf|\int_{2B(z,r)}\widetilde{k}_{\beta}(x,y)
\lf[f(y)-P_{2B(z,r)}^{(s)}(f)(y)\r]\,dy
\r|^{\frac{nq_1}{n-\beta q_1}}\,dx\r\}^{\frac{n-\beta q_1}{nq_1}}\\
&\quad\lesssim\lf|B(z,r)\r|^{-\alpha-\frac{1}{q_1}}
\lf[\int_{2B(z,r)}\lf|f(y)-P_{2B(z,r)}^{(s)}(f)(y)
\r|^{q_1}\,dy\r]^{\frac{1}{q_1}}\\
&\quad\sim\lf|2B(z,r)\r|^{-\alpha}
\lf[\fint_{2B(z,r)}\lf|f(y)
-P_{2B(z,r)}^{(s)}(f)(y)\r|^{q_1}\,dy\r]^{\frac{1}{q_1}},
\end{align*}
where $\widetilde{k}_{\beta}$ is as in \eqref{k-w}.
This shows that \eqref{A1-1} holds true in this case.

\emph{Case iii)} $q_1\in[\frac{n}{\beta},\fz)$ and $q_2\in[1,\fz)$.
In this case, choose $q_3\in (1,\frac{n}{\beta})$
such that $q_2\in[1,\frac{nq_3}{n-\beta q_3}]$.
Then, by the H\"older inequality,
Lemma \ref{fractional}(ii) with $b:=q_3$,
$\widetilde{b}:=\frac{nq_3}{n-\beta q_3}$,
and $g$ replaced by $[f-P_{2B(z,r)}^{(s)}(f)]\mathbf{1}_{2B(z,r)}$,
and $q_1>q_3$, we find that
\begin{align*}
&\lf|B(z,r)\r|^{-\alpha-\frac{\beta}{n}}
\lf[\fint_{B(z,r)}\lf|F_{B(z,r)}^{(1)}(x)\r|^{q_2}\,dx\r]^{\frac{1}{q_2}}\\
&\quad\leq\lf|B(z,r)\r|^{-\alpha-\frac{\beta}{n}}
\lf[\fint_{B(z,r)}\lf|F_{B(z,r)}^{(1)}(x)
\r|^{\frac{nq_3}{n-\beta q_3}}\,dx\r]^{\frac{n-\beta q_3}{nq_3}}\\
&\quad=\lf|B(z,r)\r|^{-\alpha-\frac{\beta}{n}}
\lf\{\fint_{B(z,r)}\lf|\int_{2B(z,r)}\widetilde{k}_{\beta}(x,y)
\lf[f(y)-P_{2B(z,r)}^{(s)}(f)(y)\r]\,dy
\r|^{\frac{nq_3}{n-\beta q_3}}\,dx\r\}^{\frac{n-\beta q_3}{nq_3}}\\
&\quad\ls\lf|B(z,r)\r|^{-\alpha}
\lf[\fint_{2B(z,r)}\lf|f(y)-P_{2B(z,r)}^{(s)}(f)(y)
\r|^{q_3}\,dy\r]^{\frac{1}{q_3}}\\
&\quad\lesssim\lf|B(z,r)\r|^{-\alpha}
\lf[\fint_{2B(z,r)}\lf|f(y)-P_{2B(z,r)}^{(s)}(f)(y)
\r|^{q_1}\,dy\r]^{\frac{1}{q_1}},
\end{align*}
which shows that \eqref{A1-1} holds true in this case,
and hence \eqref{A1-1} holds true in Cases i), ii), and iii).

Now, we show \eqref{A1}.
Indeed, from \eqref{A1-1}
and Lemma \ref{AC}, we deduce that
\begin{align*}
\sup_{r\in(0,\fz)}\mathrm{Z}_1
&\lesssim\sup_{r\in(0,\fz)}\lf[\int_{\rn}
\lf\{\lf|2B(z,r)\r|^{-\alpha}\lf[\fint_{2B(z,r)}
\lf|f(x)-P_{2B(z,r)}^{(s)}(f)(x)\r|^{q_1}\,dx
\r]^{\frac{1}{q_1}}\r\}^{p}\,dz\r]^\frac{1}{p}\\
&\sim\|f\|_{JN_{(p,q_1,s)_\alpha}^{\mathrm{con}}(\rn)}.
\end{align*}
This shows that \eqref{A1} holds true.

Finally, we estimate $\mathrm{Z}_2$.
Indeed, by \eqref{I-E-1}, \eqref{JN-I-E}, the Minkowski inequality,
Lemma \ref{AC}, and $\alpha\in(-\fz,\frac{s+\dz-\beta}{n})$, we conclude that
\begin{align*}
\mathrm{Z}_2
&\lesssim\lf\{\int_{\rn}\lf[|B(z,r)|^{-\alpha-\frac{\beta}{n}}
r^{s+\dz}\int_{\rn\setminus 2B(z,r)}
\frac{|f(y)-P_{2B(z,r)}^{(s)}(f)(y)|}{|y-z|^{n+s+\dz-\beta}}\,dy
\r]^{p}\,dz\r\}^\frac{1}{p}\\
&\lesssim \lf[\int_{\rn}\lf\{\sum_{k=1}^{\fz}
\lf(2^kr\r)^{-s-\dz+\beta}
r^{s+\dz}\lf|B(z,r)\r|^{-\alpha-\frac{\beta}{n}}\r.\r.\\
&\quad\lf.\lf.\times\lf[\fint_{2^{k+1}B(z,r)}
\lf|f(y)-P_{2^{k+1}B(z,r)}^{(s)}(f)(y)\r|^{q_1}
\,dy\r]^{\frac{1}{q_1}}\,dz\r\}^{p}\r]^\frac{1}{p}\\
&\lesssim\sum_{k=1}^{\fz}2^{k(-s-\dz+\beta)}\lf[\int_{\rn}\Bigg\{
\lf|B(z,r)\r|^{-\alpha}\r.\\
&\quad\lf.\times\lf[\fint_{2^{k+1}B(z,r)}
\lf|f(y)-P_{2^{k+1}B(z,r)}^{(s)}(f)(y)\r|^{q_1}
\,dy\r]^{\frac{1}{q_1}}\Bigg\}^{p}\,dz\r]^\frac{1}{p}\\
&\sim\sum_{k=1}^{\fz}2^{k(-s-\dz+\beta+\alpha n)}\lf[\int_{\rn}
\Bigg\{\lf|2^{k+1}B(z,r)\r|^{-\alpha}\r.\\
&\quad\lf.\times\lf[\fint_{2^{k+1}B(z,r)}\lf|f(y)
-P_{2^{k+1}B(z,r)}^{(s)}(f)(y)\r|^{q_1}\,dy\r]^{\frac{1}{q_1}}
\Bigg\}^{p}\,dz\r]^\frac{1}{p}\\
&\ls \|f\|_{JN_{(p,q_1,s)_\alpha}^{\mathrm{con}}(\rn)},
\end{align*}
where the implicit positive constants are independent of $r$.
From this, \eqref{Iw-JN-01}, \eqref{A1},
and Lemma \ref{AC}, we deduce that, for any
$f\in JN_{(p,q_1,s)_\alpha}^{\mathrm{con}}(\rn)$,
\begin{align*}
\lf\|\widetilde{I}_{\beta}(f)
\r\|_{JN_{(p,q_2,s)_{\alpha+\beta/n}}^{\mathrm{con}}(\rn)}
\ls \|f\|_{JN_{(p,q_1,s)_{\alpha}}^{\mathrm{con}}(\rn)}.
\end{align*}
This finishes the proof of that (ii) $\Rightarrow$ (i)
and hence of Theorem \ref{Iw-JN-bound}.
\end{proof}

\begin{rem}\label{rem-dual-I}
To the best of our knowledge,
Theorem \ref{Iw-JN-bound} is new even for the Campanato space
$\mathcal{C}_{\alpha,q,s}(\mathbb{R}^n)=JN_{(\fz,q,s)_\alpha}^{\mathrm{con}}(\rn)$ with $s\in\mathbb{N}$.	
Moreover, letting $p=\fz$, $s=0$, and $I_{\beta}:=\mathcal{I}_{\beta}$
as in \eqref{cla-I} with $\beta\in(0,1)$,
then, in this case, Theorem \ref{Iw-JN-bound}
is a special case of \cite[Theorem 5.1]{N10}.
\end{rem}

The following conclusion is an immediate
corollary of Theorem \ref{Iw-JN-bound};
we omit the details here.

\begin{cor}\label{Iw-JN-bound'}
Let $p\in[1,\fz]$, $q\in[1,\fz)$, $s\in\zz_+$, $\dz\in(0,1]$,
$\beta\in (0,\dz)$, $\alpha\in (-\fz,\frac{s+\dz-\beta}{n})$,
$k_{\beta}$ be an $s$-order fractional kernel with regularity $\dz$,
and $I_\beta$ the fractional integral with kernel $k_{\beta}$.
Let $\widetilde{k}_{\beta}$ be the adjoint kernel of $k_{\beta}$ as in \eqref{k-w},
and $\widetilde{I}_{\beta}$ as in Definition \ref{I-w} and Remark \ref{Rem-Iw}
with kernel $\widetilde{k}_{\beta}$. Then
$\widetilde{I}_{\beta}$ is bounded from
$JN_{(p,q,s)_{\alpha}}^{\mathrm{con}}(\rn)$ to
$JN_{(p,q,s)_{\alpha+\beta/n}}^{\mathrm{con}}(\rn)$, namely,
there exists a positive constant $C$ such that,
for any $f\in JN_{(p,q,s)_\alpha}^{\mathrm{con}}(\rn)$,
\begin{align*}
\lf\|\widetilde{I}_{\beta}(f)
\r\|_{JN_{(p,q,s)_{\alpha+\beta/n}}^{\mathrm{con}}(\rn)}
\leq C\|f\|_{JN_{(p,q,s)_{\alpha}}^{\mathrm{con}}(\rn)}
\end{align*}
if and only if $I_{\beta}$ has vanishing moments up to order $s$.
\end{cor}

Moreover, recall that, for any given $v\in(0,1)$
and for any $\{a_j\}_{j\in\zz}\subset \cc$,
\begin{align}\label{lv}
\lf(\sum_{j\in\zz}|a_j|\r)^{v}\leq \sum_{j\in\zz}|a_j|^{v};
\end{align}
see, for instance, \cite[Proposition 1.5]{stein2011}.
Using this and Corollary \ref{Iw-JN-bound'},
we have the following conclusion.

\begin{cor}\label{Iw-JN-bound''}
Let $\dz\in(0,1]$, $\beta\in (0,\dz)$,
$p\in[1,\frac{n}{\beta})$,
$\frac{1}{\widetilde{p}}:=\frac{1}{p}-\frac{\beta}{n}$,
$q\in[1,\fz)$, $s\in\zz_+$,
$\alpha\in (-\fz,\frac{s+\dz-\beta}{n})$,
$k_{\beta}$ be an $s$-order fractional kernel with regularity $\dz$,
and $I_\beta$ the fractional integral with kernel $k_{\beta}$.
Let $\widetilde{k}_{\beta}$ be the adjoint kernel of $k_{\beta}$ as in \eqref{k-w},
and $\widetilde{I}_{\beta}$ as in Definition \ref{I-w} and Remark \ref{Rem-Iw}
with kernel $\widetilde{k}_{\beta}$. Then
$\widetilde{I}_{\beta}$ is bounded from
$JN_{(p,q,s)_{\alpha}}^{\mathrm{con}}(\rn)$ to
$JN_{(\widetilde{p},q,s)_{\alpha}}^{\mathrm{con}}(\rn)$, namely,
there exists a positive constant $C$ such that,
for any $f\in JN_{(p,q,s)_\alpha}^{\mathrm{con}}(\rn)$,
\begin{align*}
\lf\|\widetilde{I}_{\beta}(f)
\r\|_{JN_{(\widetilde{p},q,s)_{\alpha}}^{\mathrm{con}}(\rn)}
\leq C\|f\|_{JN_{(p,q,s)_{\alpha}}^{\mathrm{con}}(\rn)}
\end{align*}
if and only if $I_{\beta}$ has the vanishing moments up to order $s$.
\end{cor}

\begin{proof}
Let $\dz$, $\beta$, $p$, $q$, $\widetilde{p}$, $s$, $\alpha$,
$k_{\beta}$, $I_{\beta}$, $\widetilde{k}_{\beta}$, and
$\widetilde{I}_{\beta}$ be as in the present corollary.
We first show the necessity.
Indeed, if $\widetilde{I}_{\beta}$
is bounded from $JN_{(p,q,s)_{\alpha}}^{\mathrm{con}}(\rn)$
to $JN_{(\widetilde{p},q,s)_{\alpha}}^{\mathrm{con}}(\rn)$,
then, for any $\gamma\in\zz_+^n$ with $|\gamma|\leq s$,
$$
\lf\|\widetilde{I}_{\beta}(x^{\gamma})
\r\|_{JN_{(\widetilde{p},q,s)_{\alpha}}^{\mathrm{con}}(\rn)}
\ls \lf\|x^{\gamma}\r\|_{JN_{(p,q,s)_{\alpha}}^{\mathrm{con}}(\rn)}=0,
$$
which further implies that
$\widetilde{I}_{\beta}(x^{\gamma})\in \mathcal{P}_s(\rn)$
after changing values on a set of measure zero.
By this and Proposition \ref{I-a-x}, we conclude that
$I_{\beta}$ has vanishing moments up to order $s$.
This finishes the proof of the necessity.

Now, we show the sufficiency.
Using \eqref{lv} with $v:=\frac{p}{\widetilde{p}}\in(0,1)$,
$\frac{1}{\widetilde{p}}:=\frac{1}{p}-\frac{\beta}{n}$,
and Corollary \ref{Iw-JN-bound'}, we conclude that,
for any $f\in JN_{(p,q,s)_{\alpha}}^{\mathrm{con}}(\rn)$
and any interior pairwise disjoint subcubes $\{Q_{j}\}_j$ of $\rn$
with the same edge length,
\begin{align*}
&\lf[\sum_{j}|Q_j|\lf\{|Q_j|^{-\alpha}\lf[\fint_{Q_j}
\lf|\widetilde{I}_{\beta}(f)(x)
-P_{Q_j}^{(s)}\lf(\widetilde{I}_{\beta}(f)\r)(x)\r|^q\,dx\r]^\frac{1}{q}
\r\}^{\widetilde{p}}\r]^{\frac{1}{\widetilde{p}}}\\
&\quad\leq\lf[\sum_{j}|Q_j|^{\frac{p}{\widetilde{p}}}
\lf\{|Q_j|^{-\alpha}\lf[\fint_{Q_j}
\lf|\widetilde{I}_{\beta}(f)(x)
-P_{Q_j}^{(s)}\lf(\widetilde{I}_{\beta}(f)\r)(x)\r|^q\,dx\r]^\frac{1}{q}
\r\}^{p}\r]^{\frac{1}{p}}\\
&\quad=\lf[\sum_{j}|Q_j|
\lf\{|Q_j|^{-\alpha-\frac{\beta}{n}}\lf[\fint_{Q_j}
\lf|\widetilde{I}_{\beta}(f)(x)
-P_{Q_j}^{(s)}\lf(\widetilde{I}_{\beta}(f)\r)(x)\r|^q\,dx\r]^\frac{1}{q}
\r\}^{p}\r]^{\frac{1}{p}}\\
&\quad\ls\lf\|\widetilde{I}_{\beta}(f)
\r\|_{JN_{(p,q,s)_{\alpha+\beta/n}}^{\mathrm{con}}(\rn)}
\ls\|f\|_{JN_{(p,q,s)_{\alpha}}^{\mathrm{con}}(\rn)}.
\end{align*}
This further implies that
\begin{align*}
\lf\|\widetilde{I}_{\beta}(f)\r
\|_{JN_{(\widetilde{p},q,s)_{\alpha}}^{\mathrm{con}}(\rn)}
\ls\|f\|_{JN_{(p,q,s)_{\alpha}}^{\mathrm{con}}(\rn)},
\end{align*}
which finishes the proof of
the sufficiency and hence of Corollary \ref{Iw-JN-bound''}.
\end{proof}

\noindent \\[4mm]

\section{Boundedness of fractional integrals
on $HK_{(p,q,s)_\alpha}^{\mathrm{con}}(\rn)$\label{S-C-Z-HK}}

In this section, we prove that $I_{\beta}$ can be extended
to a unique continuous linear operator from
$HK_{(p,q,s)_{\alpha+\beta/n}}^{\mathrm{con}}(\rn)$
to $HK_{(p,q,s)_{\alpha}}^{\mathrm{con}}(\rn)$.
To this end, we skillfully use properties of molecules of
$HK_{(p,q,s)_\alpha}^{\mathrm{con}}(\rn)$
and a criterion for the boundedness of linear operators on
$HK_{(p,q,s)_\alpha}^{\mathrm{con}}(\mathbb{R}^n)$,
obtained in \cite[Theorem 3.16]{jtyyz2},
to overcome the difficulty caused by the fact
that $\|\cdot\|_{HK_{(p,q,s)_{\alpha}}^{\mathrm{con}}(\mathbb{R}^n)}$
is not concave.

\subsection{Hardy-type spaces and molecules\label{Hardy-type}}

We first recall the following notion of the Hardy-kind space
$HK_{(p,q,s)_\alpha}^{\mathrm{con}}(\rn)$
which was introduced in \cite[Section 4]{jtyyz1}.
\begin{defn}\label{d3.2}
Let $p\in(1,\infty)$, $q\in(1,\infty]$, $s\in\zz_+$, and $\alpha\in\rr$.
A measurable function $a$ on $\rn$ is called a $(p,q,s)_\alpha$-\emph{atom}
supported in a cube $Q\subset \rn$ if
\begin{enumerate}
\item[\rm(i)] $\supp\,(a)\subset Q$;
\item[\rm(ii)] $\|a\|_{L^q(Q)}\leq|Q|^{\frac{1}{q}-\frac{1}{p}-\alpha}$;
\item[\rm(iii)] $\int_Qa(x)x^{\gamma}\,dx=0$
for any $\gamma\in\zz_+^n$ with $|\gamma|\leq s$.
\end{enumerate}
\end{defn}

Recall that, for any $\ell\in(0,\fz)$,
$\Pi_{\ell}(\rn)$ denotes the class of all collections of
interior pairwise disjoint subcubes
$\{Q_j\}_j$ of $\rn$ with edge length $\ell$.
Moreover, we use $(JN_{(p',q',s)_\alpha}^{\mathrm{con}}(\rn))^*$
to denote the dual space
of $JN_{(p',q',s)_\alpha}^{\mathrm{con}}(\rn)$,
which is defined to be the set of all
continuous linear functionals on
$JN_{(p',q',s)_\alpha}^{\mathrm{con}}(\rn)$
equipped with the weak-$\ast$ topology.

\begin{defn}\label{d3.3}
Let $p\in(1,\infty)$, $q\in(1,\infty]$, $s\in\zz_+$, and $\alpha\in\rr$.
The \emph{space of congruent $(p,q,s)_\alpha$-polymers},
$\widetilde{HK}_{(p,q,s)_\alpha}^{\mathrm{con}}(\rn)$,
is defined to be the set of all
$g\in (JN_{(p',q',s)_\alpha}^{\mathrm{con}}(\rn))^*$ satisfying that
there exist an $\ell\in(0,\fz)$,
$(p,q,s)_{\alpha}$-atoms $\{a_j\}_j$ supported, respectively, in
$\{Q_j\}_j\in \Pi_{\ell}(\rn)$, and $\{\lambda_j\}_j\subset \cc$ with $\sum_{j}|\lambda_j|^p<\infty$
such that $g=\sum_{j}\lambda_ja_j$ in
$(JN_{(p',q',s)_\alpha}^{\mathrm{con}}(\rn))^*$,
where $\frac{1}{p}+\frac{1}{p'}=1=\frac{1}{q}+\frac{1}{q'}$.
Moreover, any $g\in\widetilde{HK}_{(p,q,s)_\alpha}^{\mathrm{con}}(\rn)$
is called a \emph{congruent $(p,q,s)_\alpha$-polymer} with its norm
$\|g\|_{\widetilde{HK}_{(p,q,s)_\alpha}^{\mathrm{con}}(\rn)}$ defined by setting
$$\|g\|_{\widetilde{HK}_{(p,q,s)_\alpha}^{\mathrm{con}}(\rn)}
:=\inf\lf(\sum_{j}|\lambda_j|^p\r)^{\frac{1}{p}},$$
where the infimum is taken over all decompositions of $g$ as above.
\end{defn}

\begin{defn}\label{d3.6}
Let $p\in(1,\infty)$, $q\in(1,\infty]$, $s\in\zz_+$, and $\alpha\in\rr$.
The \emph{Hardy-type space via congruent cubes},
$HK_{(p,q,s)_\alpha}^{\mathrm{con}}(\rn)$, is defined by setting
\begin{align*}
HK_{(p,q,s)_\alpha}^{\mathrm{con}}(\rn)
:=\lf\{g\in(JN_{(p',q',s)_\alpha}^{\mathrm{con}}(\rn))^*
:\ g=\sum_ig_i\text{ in }(JN_{(p',q',s)_\alpha}^{\mathrm{con}}(\rn))^*,\r.\\
\quad\lf.\{g_i\}_i\subset\widetilde{HK}_{(p,q,s)_\alpha}^{\mathrm{con}}(\rn),
\text{ and }\sum_i\|g_i
\|_{\widetilde{HK}_{(p,q,s)_\alpha}^{\mathrm{con}}(\rn)}<\infty\r\},
\end{align*}
where $\frac{1}{p}+\frac{1}{p'}=1=\frac{1}{q}+\frac{1}{q'}$.
Moreover, for any $g\in HK_{(p,q,s)_\alpha}^{\mathrm{con}}(\rn)$, let
$$\|g\|_{HK_{(p,q,s)_\alpha}^{\mathrm{con}}(\rn)}
:=\inf\sum_i\|g_i\|_{\widetilde{HK}_{(p,q,s)_\alpha}^{\mathrm{con}}(\rn)},$$
where the infimum is taken over all decompositions of $g$ as above.
\end{defn}

\begin{defn}\label{d3.5}
Let $p\in(1,\infty)$, $q\in(1,\infty]$, $s\in\zz_+$, and $\alpha\in\rr$.
The \emph{finite atomic Hardy-type space via congruent cubes},
$HK_{(p,q,s)_\alpha}^{\mathrm{con-fin}}(\rn)$, is defined to be the set of all
\begin{align*}
g=\sum_{j=1}^M\lambda_ja_j
\end{align*}
pointwisely, where $M\in\nn$, $\{a_j\}_{j=1}^M$
are $(p,q,s)_{\alpha}$-atoms supported, respectively,
in cubes $\{Q_j\}_{j=1}^M$ of $\rn$,
and $\{\lambda_j\}_{j=1}^M\subset \mathbb{C}$.
\end{defn}

The following conclusion is an immediate
corollary of \cite[Theorem 4.10]{jtyyz1}
and \cite[Proposition 2.7]{jtyyz1}; we omit the details here.

\begin{lem}\label{t3.9}
Let $p$, $q\in(1,\infty)$,
$\frac{1}{p}+\frac{1}{p'}=1=\frac{1}{q}+\frac{1}{q'}$,
$s\in\zz_+$, and $\alpha\in\rr$.
Then $$\lf(HK_{(p',q',s)_{\alpha}}^{\mathrm{con}}(\rn)\r)^*
=JN_{(p,q,s)_{\alpha}}^{\mathrm{con}}(\rn)$$
with equivalent norms in the following sense:
\begin{enumerate}
\item[\rm (i)] any given
$f\in JN_{(p,q,s)_{\alpha}}^{\mathrm{con}}(\rn)$
induces a linear functional $\mathcal{L}_f$ given by setting,
for any $g\in HK_{(p',q',s)_{\alpha}}^{\mathrm{con}}(\rn)$
and $\{g_i\}_i\subset
\widetilde{HK}_{(p',q',s)_{\alpha}}^{\mathrm{con}}(\rn)$
with $g=\sum_i g_i$ in $(JN_{(p,q,s)_{\alpha}}^{\mathrm{con}}(\rn))^*$,
\begin{align*}
\mathcal{L}_f(g):=\langle g,f\rangle
=\sum_i\langle g_i,f\rangle.
\end{align*}
Moreover,
for any $g\in HK_{(p',q',s)_{\alpha}}^{\mathrm{con-fin}}(\rn)$,
$\mathcal{L}_f(g)=\int_{\rn}f(x)g(x)\,dx$
and there exists a positive constant $C$ such that
\begin{align*}
\lf\|\mathcal{L}_f\r
\|_{(HK_{(p',q',s)_{\alpha}}^{\mathrm{con}}(\rn))^*}
\leq C\|f\|_{JN_{(p,q,s)_{\alpha}}^{\mathrm{con}}(\rn)};
\end{align*}
\item[\rm (ii)] conversely, for any continuous
linear functional $\mathcal{L}$ on
$HK_{(p',q',s)_{\alpha}}^{\mathrm{con}}(\rn)$,
there exists a unique
$f\in JN_{(p,q,s)_{\alpha}}^{\mathrm{con}}(\rn)$ such that,
for any $g\in HK_{(p',q',s)_{\alpha}}^{\mathrm{con-fin}}(\rn)$,
$\mathcal{L}(g)=\int_{\rn}f(x)g(x)\,dx$
and there exists a positive constant $C$ such that
$$
\|f\|_{JN_{(p,q,s)_{\alpha}}^{\mathrm{con}}(\rn)}
\leq C\|\mathcal{L}\|_{(HK_{(p',q',s)_{\alpha}}^{\mathrm{con}}(\rn))^*}.
$$
\end{enumerate}
\end{lem}

Also, recall the following notion of the
$(p,q,s,\alpha,\epsilon)$-molecule
of $HK_{(p,q,s)_\alpha}^{\mathrm{con}}(\rn)$
introduced in \cite[Definition 3.10]{jtyyz2}.
In what follows, for any $z\in\rn$ and $r\in(0,\fz)$,
$Q_z(r)$ denotes the cube with center $z$ and edge length $r$.

\begin{defn}\label{Def-mole}
Let $p\in (1,\fz)$, $q\in(1,\fz]$, $s\in\zz_+$,
$\alpha\in(\frac{1}{q}-\frac{1}{p},\fz)$,
and $\epsilon\in (0,\fz)$.
A measurable function $M$ on $\rn$ is called
a \textit{$(p,q,s,\alpha,\epsilon)$-molecule}
centered at the cube $Q_z(r)$ with center $z\in\rn$
and edge length $r\in(0,\fz)$ if
\begin{enumerate}
\item[\rm (i)] $\|M\mathbf{1}_{Q_z(r)}\|_{L^q(\rn)}
\leq |Q_z(r)|^{\frac{1}{q}-\frac{1}{p}-\alpha}$;
\item[\rm (ii)] for any $j\in \nn$, $\|M\mathbf{1}_{Q_z(2^{j}r)
\setminus Q_z(2^{j-1}r)}\|_{L^q(\rn)}
\leq 2^{\frac{jn}{\epsilon}(\frac{1}{q}-\frac{1}{p}-\alpha)}
|Q_z(r)|^{\frac{1}{q}-\frac{1}{p}-\alpha}$;
\item[\rm (iii)] $\int_{\rn}M(x)x^{\gamma}\,dx=0$
for any $\gamma\in\zz_+^n$ with $|\gamma|\leq s$.
\end{enumerate}
\end{defn}

Next, we recall two crucial lemmas,
namely, Lemmas \ref{M-JN1} and \ref{Bounded-HK-A} below,
which were proved in \cite[Propositions 3.11 and 3.15, and Theorem 3.16]{jtyyz2}. The first lemma
shows that a $(p,q,s,\alpha,\epsilon)$-molecule
induces an element of $(JN_{(p',q',s)_\alpha}^{\mathrm{con}}(\rn))^*$.

\begin{lem}\label{M-JN1}
Let $p\in (1,\fz)$, $q\in(1,\fz]$,
$\frac{1}{p}+\frac{1}{p'}=1=\frac{1}{q}+\frac{1}{q'}$, $s\in\zz_+$,
$\alpha\in(\frac{1}{q}-\frac{1}{p},\fz)$,
$\epsilon\in(0,1)$ be such that
$\frac{1}{\epsilon}
(\frac{1}{q}-\frac{1}{p}-\alpha)+\frac{1}{q'}+\frac{s}{n}<0$,
and $M$ a $(p,q,s,\alpha,\epsilon)$-molecule.
Then, for any $f\in JN_{(p',q',s)_\alpha}^{\mathrm{con}}(\rn)$,
$fM$ is integrable,
$$\langle M,f\rangle:=\int_{\rn}f(x)M(x)\,dx$$
induces a continuous linear functional on
$JN_{(p',q',s)_\alpha}^{\mathrm{con}}(\rn)$,
$M\in HK_{(p,q,s)_\alpha}^{\mathrm{con}}(\rn)$,
and $\|M\|_{HK_{(p,q,s)_\alpha}^{\mathrm{con}}(\rn)}\leq C$,
where the positive constant $C$ is independent of $M$.
\end{lem}

The second lemma is a useful
criterion for the boundedness of linear operators on
the Hardy-kind space $HK_{(p,q,s)_\alpha}^{\mathrm{con}}(\rn)$.

\begin{lem}\label{Bounded-HK-A}
Let $p_1$, $p_2$, $q_1$, $q_2\in (1,\fz)$,
$\frac{1}{p_i}+\frac{1}{p_i'}=1
=\frac{1}{q_i}+\frac{1}{q_i'}$ for any $i\in\{1,2\}$,
$s_1$, $s_2\in\zz_+$, $\alpha_1\in\rr$,
and $\alpha_2\in (\frac{1}{q_2}
-\frac{1}{p_2},\fz)$.
Let $A$ be a linear operator defined on
$HK_{(p_1,q_1,s_1)_{\alpha_1}}^{\mathrm{con-fin}}(\rn)$,
and $\widetilde{A}$ a linear operator bounded
from $JN_{(p_2',q_2',s_2)_{\alpha_2}}^{\mathrm{con}}(\rn)$
to $JN_{(p_1',q_1',s_1)_{\alpha_1}}^{\mathrm{con}}(\rn)$.
If the following statements hold true:
\begin{enumerate}
\item[\rm (i)] there exists a positive constant $C_0$ such that,
for any $(p_1,q_1,s_1)_{\alpha_1}$-atom $a$
and some given $\epsilon\in(0,1)$ such that
$\frac{1}{\epsilon}
(\frac{1}{q_2}-\frac{1}{p_2}-\alpha_2)+\frac{1}{q_2'}+\frac{s_2}{n}<0$,
$A(a)/C_0$ is a $(p_2,q_2,
s_2,\alpha_2,\epsilon)$-molecule;
\item[\rm (ii)] for any $(p_1,q_1,s_1)_{\alpha_1}$-atom $a$ and any
$f\in JN_{(p_2',q_2',s_2)_{\alpha_2}}^{\mathrm{con}}(\rn)$,
$\langle A(a),f\rangle=\langle a, \widetilde{A}(f)\rangle$,
\end{enumerate}
then the linear operator $A$ has a unique continuous linear extension,
still denoted by $A$, from
$HK_{(p_1,q_1,s_1)_{\alpha_1}}^{\mathrm{con}}(\rn)$
to $HK_{(p_2,q_2,s_2)_{\alpha_2}}^{\mathrm{con}}(\rn)$,
namely, there exists a positive constant $C$ such that,
for any $g\in HK_{(p_1,q_1,s_1)_{\alpha_1}}^{\mathrm{con}}(\rn)$,
$$\|A(g)\|_{HK_{(p_2,q_2,s_2)_{\alpha_2}}^{\mathrm{con}}(\rn)}
\leq C\|g\|_{HK_{(p_1,q_1,s_1)_{\alpha_1}}^{\mathrm{con}}(\rn)}$$
and, moreover,
for any $f\in JN_{(p_2',q_2',s_2)_{\alpha_2}}^{\mathrm{con}}(\rn)$,
\begin{align*}
\langle A(g),f\rangle
=\lf\langle g,\widetilde{A}(f)\r\rangle.
\end{align*}
\end{lem}

\subsection{Fractional integrals
on $HK_{(p,q,s)_\alpha}^{\mathrm{con}}(\rn)$\label{S-I-HK}}

In this subsection, we show that $I_{\beta}$ in Definition \ref{Def-I-s}
can be extended to a unique continuous linear operator from
$HK_{(p,q,s)_{\alpha+\beta/n}}^{\mathrm{con}}(\rn)$
to $HK_{(p,q,s)_\alpha}^{\mathrm{con}}(\rn)$.
To this end, in Lemmas \ref{Lem-I-HK} and \ref{Ia-m}
below, we first prove that $I_{\beta}$
and $\widetilde{I}_{\beta}$ satisfy (i) and (ii) of Lemma \ref{Bounded-HK-A}, 
respectively.

\begin{lem}\label{Lem-I-HK}
Let $p$, $q\in(1,\fz)$,
$\frac{1}{q}+\frac{1}{q'}=1$, $s\in\zz_+$,
$\alpha\in(\frac{1}{q}-\frac{1}{p},\fz)$,
$\dz\in(0,1]$, $\beta\in(0,\dz)$,
and $\epsilon\in(0,1)$ be such that
$-\frac{1}{q'}-\frac{s+\dz-\beta}{n}
\leq \frac{1}{\epsilon}(\frac{1}{q}-\frac{1}{p}-\alpha)$.
Let $k_{\beta}$ be an $s$-order fractional kernel with regularity $\dz$
and $I_{\beta}$ the fractional integral with kernel $k_{\beta}$.
Then the following two statements are equivalent:
\begin{enumerate}
\item[\rm (i)]
there exists a positive constant $C$
such that, for any $(p,q,s)_{\alpha+\beta/n}$-atom
$a$, $I_\beta (a)/C$
is a $(p,q,s,\alpha,\epsilon)$-molecule;
\item[\rm (ii)]
$I_{\beta}$ has the vanishing moments up to order $s$
as in Definition \ref{Def-I-s}.
\end{enumerate}
\end{lem}

\begin{proof}
Let $\beta$, $\dz$, $p$, $q$, $q'$, $s$, $\alpha$, $k_{\beta}$, and
$I_{\beta}$ be as in the present lemma.
We first show that (i) $\Rightarrow$ (ii).
Indeed, let $h\in L_s^q(\rn)$ be supported in a
ball $B(x_0,r_0)$ with $x_0\in\rn$ and $r_0\in(0,\fz)$.
Without loss of generality, we
may assume that $\|h\|_{L^q(\rn)}>0$.
Then $\widetilde{h}:=\frac{|Q_{x_0}(2r_0)
|^{\frac{1}{q}-\frac{1}{p}-\alpha-\frac{\beta}{n}}h}
{\|h\|_{L^q(\rn)}}$ is a $(p,q,s)_{\alpha+\beta/n}$-atom
supported in $Q_{x_0}(2r_0)$.
From this, (i) of the present lemma, and \eqref{0-0}, we deduce that,
for any $\gamma\in\zz_+^n$ with $|\gamma|\leq s$,
$$
\int_{B(x_0,r_0)}\widetilde{h}(x)\widetilde{I}_{\beta}(y^{\gamma})(x)\,dx
=\int_{\rn}I_{\beta}(\widetilde{h})(x)x^{\gamma}\,dx
=0.
$$
By this and a proof similar to that of Proposition \ref{I-a-x},
we conclude that, for any $\gamma\in\zz_+^n$ with $|\gamma|\leq s$,
$\widetilde{I}_{\beta}(y^{\gamma})\in\mathcal{P}_s(\rn)$
after changing values on a set of measure zero.
From this and Proposition \ref{I-a-x}, we deduce that
$I_{\beta}$ has the vanishing moments up to order $s$.
This finishes the proof of that (i) $\Rightarrow$ (ii).

Now, we show that (ii) $\Rightarrow$ (i).
Let $a$ be a $(p,q,s)_{\alpha+\beta/n}$-atom
supported in the cube $Q_{z}(r)$
with $z\in \rn$ and $r\in (0,\fz)$.
To show that $I_\beta (a)$
is a $(p,q,s,\alpha,\epsilon)$-molecule,
we first claim that there exists a
positive constant $C_1$, independent of $a$,
such that
\begin{align}\label{Ia-i}
\lf\|I_\beta(a)\mathbf{1}_{Q_z(R_0)}\r\|_{L^q(\rn)}
\leq C_1\lf|Q_z(R_0)\r|^{\frac{1}{q}-\frac{1}{p}-\alpha}
\end{align}
with $R_0:=2\sqrt n r$.
Indeed, from \eqref{size-i'} with $\gamma:=\mathbf{0}$,
and $\supp\,(a)\subset Q_z(r)$,
we deduce that, for any $x\in Q_z(R_0)$,
\begin{align}\label{I-i}
\lf|I_\beta (a)(x)\r|
&\leq\int_{Q_z(r)}\lf|k_{\beta}(x,y)a(y)\r|\,dy\noz\\
&\ls\int_{Q_z(r)}\frac{|a(y)|}{|x-y|^{n-\beta}}\,dy
\ls\int_{Q(x,2R_0)}\frac{|a(y)|}{|x-y|^{n-\beta}}\,dy\noz\\
&\ls\sum_{j\in\zz_+}
\int_{Q(x,2^{-j+1}R_0)\setminus Q(x,2^{-j}R_0)}
\frac{|a(y)|}{|x-y|^{n-\beta}}\,dy\noz\\
&\lesssim\sum_{j\in\zz_+}\lf(2^{-j+1}R_0\r)^{\beta-n}
\int_{Q(x,2^{-j+1}R_0)\setminus Q(x,2^{-j}R_0)}|a(y)|\,dy\noz\\
&\sim R_0^{\beta}\sum_{j\in\zz_+}2^{-j\beta}
\fint_{Q(x,2^{-j+1}R_0)}|a(y)|\,dy
\lesssim R_0^{\beta}\mathcal{M}(a)(x),
\end{align}
where $\mathcal{M}$ is the \textit{Hardy--Littlewood maximal operator}
defined by setting,
for any $f\in L^{1}_{\mathrm{loc}}(\rn)$ and $x\in\rn$,
$$\mathcal{M}(f)(x):=\sup_{Q\ni x}\fint_{Q}|f(y)|\,dy$$
with the supremum taken over all cubes containing $x$.
Then, by \eqref{I-i}, the boundedness of $\mathcal{M}$
on $L^q(\rn)$ (see, for instance, \cite[p.\,31, Theorem 2.5]{Duo01}),
Definition \ref{d3.2}(ii), and $r\sim R_0$, we conclude that
\begin{align*}
\lf\|I_\beta (a)\mathbf{1}_{Q_z(R_0)}\r\|_{L^q(\rn)}
&\lesssim R_0^{\beta}\lf\|\mathcal{M}(a)\r\|_{L^q(\rn)}
\lesssim R_0^{\beta}\lf\|a\r\|_{L^q(\rn)}\\
&\lesssim R_0^{\beta}\lf|Q_z(r)
\r|^{\frac{1}{q}-\frac{1}{p}-\alpha-\frac{\beta}{n}}
\sim \lf|Q_z(R_0)\r|^{\frac{1}{q}-\frac{1}{p}-\alpha}.
\end{align*}
This implies that \eqref{Ia-i} and hence
the above claim hold true.

Next, we show that there exists a
positive constant $C_2$, independent of $a$,
such that, for any $j\in\nn$,
\begin{align}\label{Ia-ii}
\lf\|I_\beta (a)\mathbf{1}_{Q_z(2^{j}R_0)
\setminus Q_z(2^{j-1}R_0)}\r\|_{L^q(\rn)}
\leq C_2 2^{\frac{jn}{\epsilon}(\frac{1}{q}-\frac{1}{p}-\alpha)}
|Q_z(R_0)|^{\frac{1}{q}-\frac{1}{p}-\alpha}.
\end{align}
Indeed, from the fact $a\in L^q_s(\rn)$, the Taylor remainder theorem,
the H\"older inequality, and Definition \ref{d3.2}(ii),
we deduce that, for any $j\in\nn$,
$x\in Q_z(2^{j}R_0)\setminus Q_z(2^{j-1}R_0)$, $y\in Q_z(r)$,
there exists a $\widetilde{y}\in Q_z(r)$ such that
\begin{align*}
\lf|I_\beta (a)(x)\r|
&=\lf|\int_{Q_z(r)}k_{\beta}(x,y)a(y)\,dy\r|\\
&=\lf|\int_{Q_z(r)}\lf[k_{\beta}(x,y)
-\sum_{\{\gamma\in\zz_+^n:\ |\gamma|\leq s\}}
\frac{\partial_{(2)}^{\gamma}k_{\beta}(x,z)}{\gamma!}(y-z)^{\gamma}
\r]a(y)\,dy\r|\\
&=\lf|\int_{Q_z(r)}\lf[
\sum_{\{\gamma\in\zz_+^n:\ |\gamma|=s\}}
\frac{\partial_{(2)}^{\gamma}k_{\beta}(x,\widetilde{y})
-\partial_{(2)}^{\gamma}k_{\beta}(x,z)}{\gamma!}(y-z)^{\gamma}
\r]a(y)\,dy\r|\\
&\lesssim\int_{Q_z(r)}\frac{
|\widetilde{y}-z|^\dz|y-z|^{s}|a(y)|}{|x-z|^{n+s+\dz-\beta}}\,dy
\lesssim\frac{|Q_z(r)|^{\frac{1}{q'}}
r^{s+\dz}}{|x-z|^{n+s+\dz-\beta}}
\|a\|_{L^{q}(\rn)}\\
&\ls \frac{r^{\frac{n}{q'}+s+\dz}}{(2^{j}r)^{n+s+\dz-\beta}}
\lf|Q_z(r)\r|^{\frac{1}{q}-\frac{1}{p}-\alpha-\frac{\beta}{n}}
\sim 2^{-jn(1+\frac{s+\dz-\beta}{n})}
\lf|Q_z(r)\r|^{-\frac{1}{p}-\alpha},
\end{align*}
where, in the fourth step, we used \eqref{regular-i'}
together with $|x-z|\geq 2|\widetilde{y}-z|$.
By this, $-\frac{1}{q'}-\frac{s+\dz-\beta}{n}
\leq \frac{1}{\epsilon}(\frac{1}{q}-\frac{1}{p}-\alpha)$,
and $r\sim R_0$, we obtain, for any $j\in\nn$,
\begin{align*}
&\lf\|I_\beta (a)\mathbf{1}_{Q_z(2^{j}R_0)
\setminus Q_z(2^{j-1}R_0)}\r\|_{L^q(\rn)}\\
&\quad\lesssim 2^{-jn(1+\frac{s+\dz-\beta}{n})}
\lf|Q_z(r)\r|^{-\frac{1}{p}-\alpha}
\lf|Q_z(2^{j}r)\setminus Q_z(2^{j-1}r)\r|^{\frac{1}{q}}\\
&\quad\lesssim 2^{jn(-1-\frac{s+\dz-\beta}{n}+\frac{1}{q})}
\lf|Q_z(r)\r|^{\frac{1}{q}-\frac{1}{p}-\alpha}
\lesssim 2^{\frac{jn}{\epsilon}(\frac{1}{q}-\frac{1}{p}-\alpha)}
\lf|Q_z(r)\r|^{\frac{1}{q}-\frac{1}{p}-\alpha}\\
&\quad\sim 2^{\frac{jn}{\epsilon}(\frac{1}{q}-\frac{1}{p}-\alpha)}
|Q_z(R_0)|^{\frac{1}{q}-\frac{1}{p}-\alpha},
\end{align*}
which implies that \eqref{Ia-ii} holds true.

To sum up, let $C:=\max\{C_1,C_2\}$.
Then, by \eqref{Ia-i}, \eqref{Ia-ii},
and \eqref{I-x-gam}, we find that $I_{\beta}(a)/C$
satisfies (i), (ii), and (iii) of Definition \ref{Def-mole}.
This finishes the proof of that (ii) $\Rightarrow$ (i)
and hence of Lemma \ref{Lem-I-HK}.
\end{proof}

\begin{lem}\label{Ia-m}
Let $p$, $q\in (1,\fz)$, $s\in\zz_+$,
$\dz\in(0,1]$, $\beta\in(0,\dz)$,
$\alpha\in(\frac{1}{q}-\frac{1}{p},
1-\frac{1}{p}+\frac{s+\dz-\beta}{n})$,
$k_{\beta}$ be an $s$-order fractional kernel with regularity $\dz$,
and $I_\beta$ a fractional integral having
the vanishing moments up to
order $s$ with kernel $k_{\beta}$. Let $\widetilde{k}_{\beta}$ be 
the adjoint kernel of $k_{\beta}$ as in
\eqref{k-w}, and $\widetilde{I}_\beta$ as in Definition \ref{I-w}
and Remark \ref{Rem-Iw} with kernel $\widetilde{k}_{\beta}$.
Then, for any $(p,q,s)_{\alpha+\beta/n}$-atom $a$
and any $f\in JN_{(p',q',s)_\alpha}^{\mathrm{con}}(\rn)$,
\begin{align}\label{4.32y}
\lf\langle I_{\beta}(a), f\r\rangle
=\lf\langle a, \widetilde{I}_{\beta}(f)\r\rangle.
\end{align}
\end{lem}

\begin{proof}
Let $p$, $q$, $s$, $\dz$, $\beta$, $\alpha$, $k_{\beta}$,
$I_{\beta}$, $\widetilde{k}_{\beta}$,
and $\widetilde{I}_{\beta}$ be as in the present lemma.
Also, let $a$ be any given $(p,q,s)_{\alpha+\beta/n}$-atom
supported in the cube $Q_{x_0}(2r_0/\sqrt n)$
with $x_0\in\rn$ and $r_0\in (0,\fz)$,
and $B_0:=B(x_0,r_0)$. Then
$Q_{x_0}(2r_0/\sqrt n)\subset B_0$.
From this and Lemmas \ref{M-JN1} and \ref{Lem-I-HK}, we deduce that,
for any $f\in JN_{(p',q',s)_\alpha}^{\mathrm{con}}(\rn)$,
\begin{align}\label{Ia-f}
&\lf\langle I_{\beta}(a), f\r\rangle\noz\\
&\quad=\int_{\rn} I_{\beta}(a)(x)f(x)\,dx
=\int_{\rn}\int_{B_0}k_\beta(x,y)a(y)\,dy\,f(x)\,dx\noz\\
&\quad=\int_{2B_0}\int_{B_0}k_\beta(x,y)a(y)\,dy\,f(x)\,dx\noz\\
&\qquad+\int_{\rn\setminus 2B_0}\int_{B_0}\lf[k_{\beta}(x,y)
-\sum_{\{\gamma\in\zz_+^n:\ |\gamma|\leq s\}}
\frac{\partial_{(2)}^{\gamma}k_{\beta}(x,x_0)}{\gamma!}(y-x_0)^{\gamma}
\r]a(y)\,dy\,f(x)\,dx\noz\\
&\quad=:\mathrm{D}_1+\mathrm{D}_2.
\end{align}
Thus, to show \eqref{4.32y}, it suffices to calculate
$\mathrm{D}_1$ and $\mathrm{D}_2$.

We first consider $\mathrm{D}_1$.
Indeed, let $\mathcal{I}_{\beta}$ be as in \eqref{cla-I}.
By \eqref{size-i'} with $\gamma:=\mathbf{0}$,
\eqref{3.13x}, \eqref{2.15x}, and the fact that
$f\mathbf{1}_{2B_0}\in L^{q'}(\rn)$,
we conclude that
\begin{align*}
\int_{2B_0}\int_{B_0}\lf|k_{\beta}(x,y)a(y)\r|\,dy\,\lf|f(x)\r|\,dx
&\ls\int_{2B_0}\int_{B_0}\frac{|a(y)|}{|x-y|^{n-\beta}}\,dy\,\lf|f(x)\r|\,dx\\
&\ls \lf\|\mathcal{I}_{\beta}(|a|)\r\|_{L^q(2B_0)}
\|f\|_{L^{q'}(2B_0)}<\fz.
\end{align*}
From this and the Fubini theorem,
we deduce that
\begin{align}\label{I-01}
\mathrm{D}_1=\int_{B_0}\int_{2B_0}k_{\beta}(x,y)a(y)f(x)\,dx\,dy.
\end{align}
This is a desired conclusion of $\mathrm{D}_1$.

Now, we consider $\mathrm{D}_2$. To this end, we first claim that
\begin{align}\label{I-HK-a-1}
\int_{\rn\setminus 2B_0}\int_{B_0}\lf|k_{\beta}(x,y)
-\sum_{\{\gamma\in\zz_+^n:\ |\gamma|\leq s\}}
\frac{\partial_{(2)}^{\gamma}k_{\beta}(x,x_0)}{\gamma!}
(y-x_0)^{\gamma}\r||a(y)||f(x)|\,dy\,dx<\fz.
\end{align}
Indeed, by \eqref{I-HK-a-1'}, we conclude that
\begin{align}\label{I-HK-a-2}
&\int_{\rn\setminus 2B_0}\int_{B_0}\lf|k_{\beta}(x,y)
-\sum_{\{\gamma\in\zz_+^n:\ |\gamma|\leq s\}}
\frac{\partial_{(2)}^{\gamma}k_{\beta}(x,x_0)}{\gamma!}
(y-x_0)^{\gamma}\r|\noz\\
&\qquad\times|a(y)|\lf|P_{2B_0}^{(s)}(f)(x)\r|\,dy\,dx\noz\\
&\quad<\fz.
\end{align}
Moreover, from the Tonelli theorem, the Taylor remainder theorem,
Lemma \ref{I-JN} with $\lz:=s+\dz-\beta\in(s,\fz)$,
and $\alpha\in(\frac{1}{q}-\frac{1}{p},
1-\frac{1}{p}+\frac{s+\dz-\beta}{n})$,
we deduce that, for any $y\in B_0$,
there exists a $\widetilde{y}\in B_0$ such that
\begin{align*}
&\int_{\rn\setminus 2B_0}\int_{B_0}\lf|k_{\beta}(x,y)
-\sum_{\{\gamma\in\zz_+^n:\ |\gamma|\leq s\}}
\frac{\partial_{(2)}^{\gamma}k_{\beta}(x,x_0)}{\gamma!}
(y-x_0)^{\gamma}\r|\noz\\
&\qquad\times|a(y)|\lf|f(x)-P_{2B_0}^{(s)}(f)(x)\r|\,dy\,dx\\
&\quad=\int_{B_0}\int_{\rn\setminus 2B_0}\lf|k_{\beta}(x,y)
-\sum_{\{\gamma\in\zz_+^n:\ |\gamma|\leq s\}}
\frac{\partial_{(2)}^{\gamma}k_{\beta}(x,x_0)}{\gamma!}
(y-x_0)^{\gamma}\r||a(y)|\\
&\qquad\times\lf|f(x)-P_{2B_0}^{(s)}(f)(x)\r|\,dx\,dy\\
&\quad=\int_{B_0}\int_{\rn\setminus 2B_0}\lf|
\sum_{\{\gamma\in\zz_+^n:\ |\gamma|=s\}}
\frac{\partial_{(2)}^{\gamma}k_{\beta}(x,\widetilde{y})
-\partial_{(2)}^{\gamma}k_{\beta}(x,x_0)}{\gamma!}
(y-x_0)^{\gamma}\r||a(y)|\\
&\qquad\times\lf|f(x)-P_{2B_0}^{(s)}(f)(x)\r|\,dx\,dy\\
&\quad\lesssim\int_{B_0}|a(y)|\int_{\rn\setminus 2B_0}
\frac{|\widetilde{y}-x_0|^{\dz}|y-x_0|^{s}
|f(x)-P_{2B_0}^{(s)}(f)(x)|}{|x-x_0|^{n+s+\dz-\beta}}\,dx\,dy\\
&\quad \lesssim r_0^{-\frac{n}{p'}+\beta+\alpha n}
\|a\|_{L^1(B_0)}\|f\|_{JN_{(p',q',s)_\alpha}^{\mathrm{con}}(\rn)}
<\fz,
\end{align*}
where, in the third step, we used \eqref{regular-i'}
together with $|x-x_0|\geq 2|\widetilde{y}-x_0|$.
Using this and \eqref{I-HK-a-2}, we conclude that
\begin{align*}
&\int_{\rn\setminus 2B_0}\int_{B_0}\lf|k_{\beta}(x,y)
-\sum_{\{\gamma\in\zz_+^n:\ |\gamma|\leq s\}}
\frac{\partial_{(2)}^{\gamma}k_{\beta}(x,x_0)}{\gamma!}
(y-x_0)^{\gamma}\r||a(y)||f(x)|\,dy\,dx\\
&\quad\leq\int_{\rn\setminus 2B_0}\int_{B_0}\lf|k_{\beta}(x,y)
-\sum_{\{\gamma\in\zz_+^n:\ |\gamma|\leq s\}}
\frac{\partial_{(2)}^{\gamma}k_{\beta}(x,x_0)}{\gamma!}
(y-x_0)^{\gamma}\r||a(y)|\lf|P_{2B_0}^{(s)}(f)(x)\r|\,dy\,dx\\
&\quad\quad+\int_{\rn\setminus 2B_0}\int_{B_0}\lf|k_{\beta}(x,y)
-\sum_{\{\gamma\in\zz_+^n:\ |\gamma|\leq s\}}
\frac{\partial_{(2)}^{\gamma}k_{\beta}(x,x_0)}{\gamma!}
(y-x_0)^{\gamma}\r||a(y)|\\
&\qquad\times\lf|f(x)-P_{2B_0}^{(s)}(f)(x)\r|\,dy\,dx\\
&\quad<\fz.
\end{align*}
This shows that \eqref{I-HK-a-1} holds true.
Furthermore, by \eqref{I-HK-a-1} and
the Fubini theorem, we find that
\begin{align}\label{I-02}
\mathrm{D}_2=\int_{B_0}a(y)\int_{\rn\setminus 2B_0}\lf[k_{\beta}(x,y)
-\sum_{\{\gamma\in\zz_+^n:\ |\gamma|\leq s\}}
\frac{\partial_{(2)}^{\gamma}k_{\beta}(x,x_0)}{\gamma!}(y-x_0)^{\gamma}
\r]f(x)\,dx\,dy.
\end{align}
To sum up, using \eqref{k-w}, \eqref{Ia-f}, \eqref{I-01},
\eqref{I-02}, Remark \ref{Rem-Iw}, and Lemma \ref{t3.9}, we have
\begin{align*}
&\lf\langle I_{\beta}(a), f\r\rangle\\
&\quad=\int_{B_0}a(y)\int_{\rn}f(x)\lf[k_{\beta}(x,y)
-\sum_{\{\gamma\in\zz_+^n:\ |\gamma|\leq s\}}
\frac{\partial_{(2)}^{\gamma}k_{\beta}(x,x_0)}{\gamma!}(y-x_0)^{\gamma}
\mathbf{1}_{\rn\setminus 2B_0}(x)\r]\,dx\,dy\\
&\quad=\int_{B_0}a(y)\int_{\rn}f(x)\lf[\widetilde{k}_{\beta}(y,x)
-\sum_{\{\gamma\in\zz_+^n:\ |\gamma|\leq s\}}\frac{\partial_{(1)}^{\gamma}
\widetilde{k}_{\beta}(x_0,x)}{\gamma!}(y-x_0)^{\gamma}
\mathbf{1}_{\rn\setminus 2B_0}(x)\r]\,dx\,dy\\
&\quad=\int_{B_0}a(y)\widetilde{I}_\beta (f)(y)\,dy
=\lf\langle a, \widetilde{I}_{\beta}(f)\r\rangle.
\end{align*}
This finishes the proof of Lemma \ref{Ia-m}.
\end{proof}

Finally, we give the main result of this section.

\begin{thm}\label{I-Bounded-HK}
Let $p$, $q\in (1,\fz)$,
$\frac{1}{p}+\frac{1}{p'}=1=\frac{1}{q}+\frac{1}{q'}$,
$s\in\zz_+$, $\dz\in(0,1]$, $\beta\in(0,\dz)$, and $\alpha\in\rr$
satisfy $\frac{1}{q}-\frac{1}{p}<\alpha<\frac{s+\dz-\beta}{n}$,
$k_{\beta}$ be the $s$-order fractional kernel with regularity $\dz$,
and $I_\beta$ the fractional integral
having the vanishing moments
up to order $s$ with kernel $k_{\beta}$.
Then the following two statements are equivalent:
\begin{enumerate}
\item[\rm (i)]
there exists a positive constant $C$ such that,
for any $(p,q,s)_{\alpha+\beta/n}$-atom $a$,
\begin{align}\label{Ia-C}
\lf\|I_\beta (a)\r\|_{HK_{(p,q,s)_\alpha}^{\mathrm{con}}(\rn)}\leq C.
\end{align}
Moreover, $I_\beta$ can be extended to a unique continuous
linear operator, still denoted by $I_\beta$,
from $HK_{(p,q,s)_{\alpha+\beta/n}}^{\mathrm{con}}(\rn)$
to $HK_{(p,q,s)_\alpha}^{\mathrm{con}}(\rn)$,
namely, there exists a positive constant
$C$ such that, for any
$g\in HK_{(p,q,s)_{\alpha+\beta/n}}^{\mathrm{con}}(\rn)$,
\begin{align}\label{Ia-c1}
\lf\|I_\beta (g)\r\|_{HK_{(p,q,s)_{\alpha}}^{\mathrm{con}}(\rn)}
\leq C\|g\|_{HK_{(p,q,s)_{\alpha+\beta/n}}^{\mathrm{con}}(\rn)}
\end{align}
and, for any $f\in JN_{(p',q',s)_\alpha}^{\mathrm{con}}(\rn)$,
$\widetilde{I}_{\beta}(f)
\in JN_{(p',q',s)_{\alpha+\beta/n}}^{\mathrm{con}}(\rn)$ and
\begin{align}\label{Ia-c2}
\lf\langle I_\beta (g),f\r\rangle
=\lf\langle g,\widetilde{I}_\beta (f)\r\rangle,
\end{align}
where $\widetilde{I}_{\beta}$ is as in Definition \ref{I-w} and Remark \ref{Rem-Iw}
with kernel $\widetilde{k}_{\beta}$ being the adjoint kernel
of $k_{\beta}$ as in \eqref{k-w};

\item[\rm (ii)]
$I_{\beta}$ has the vanishing moments up to order $s$
as in Definition \ref{Def-I-s}.
\end{enumerate}
\end{thm}

\begin{proof}
Let $p$, $q$, $p'$, $q'$, $s$, $\dz$, $\beta$, $\alpha$, $k_\beta$,
$I_{\beta}$, $\widetilde{k}_{\beta}$,
and $\widetilde{I}_{\beta}$ be as in the present lemma.
We first prove that (i) $\Rightarrow$ (ii).
From Lemma \ref{t3.9}, \eqref{Ia-c2}, and \eqref{Ia-c1},
we deduce that, for any
$f\in JN_{(p',q',s)_\alpha}^{\mathrm{con}}(\rn)$,
\begin{align*}
&\lf\|\widetilde{I}_{\beta}(f)
\r\|_{JN_{(p',q',s)_{\alpha+\beta/n}}^{\mathrm{con}}(\rn)}\\
&\quad\sim\lf\|\mathcal{L}_{\widetilde{I}_{\beta}(f)}
\r\|_{(HK_{(p,q,s)_{\alpha+\beta/n}}^{\mathrm{con}}(\rn))^*}
\sim\sup_{\|g\|_{HK_{(p,q,s)_{\alpha+\beta/n}}^{\mathrm{con}}(\rn)}=1}
\lf|\lf\langle \mathcal{L}_{\widetilde{I}_{\beta}(f)},g\r\rangle\r|\\
&\quad\sim\sup_{\|g\|_{HK_{(p,q,s)_{\alpha+\beta/n}}^{\mathrm{con}}(\rn)}=1}
\lf|\lf\langle g, \widetilde{I}_{\beta}(f)\r\rangle\r|
\sim\sup_{\|g\|_{HK_{(p,q,s)_{\alpha+\beta/n}}^{\mathrm{con}}(\rn)}=1}
\lf|\langle I_{\beta}(g), f\rangle\r|\\
&\quad\ls\sup_{\|g\|_{HK_{(p,q,s)_{\alpha+\beta/n}}^{\mathrm{con}}(\rn)}=1}
\|I_{\beta}(g)\|_{HK_{(p,q,s)_\alpha}^{\mathrm{con}}(\rn)}
\|f\|_{JN_{(p',q',s)_\alpha}^{\mathrm{con}}(\rn)}\\
&\quad\ls\|f\|_{JN_{(p',q',s)_\alpha}^{\mathrm{con}}(\rn)},
\end{align*}
where $\mathcal{L}_{\widetilde{I}_{\beta}(f)}$ is as in Lemma
\ref{t3.9} with $f$ replaced by $\widetilde{I}_{\beta}(f)$.
This shows that $\widetilde{I}_{\beta}$ is bounded from
$JN_{(p',q',s)_\alpha}^{\mathrm{con}}(\rn)$ to
$JN_{(p',q',s)_{\alpha+\beta/n}}^{\mathrm{con}}(\rn)$,
which, combined with Corollary \ref{Iw-JN-bound''},
further implies that $I_{\beta}$ has the vanishing moments up to order $s$.
This finishes the proof that (i) $\Rightarrow$ (ii).

Next, we show that (ii) $\Rightarrow$ (i).	
Since $\frac{1}{q}-\frac{1}{p}<\alpha<\frac{s+\dz-\beta}{n}$,
it follows that there exists an $\epsilon\in(0,1)$ such that
$$\frac{1}{\epsilon}
\left(\frac{1}{q}-\frac{1}{p}-\alpha\r)+\frac{1}{q'}+\frac{s}{n}<0\quad\mathrm{and}\quad
-\frac{1}{q'}-\frac{s+\dz-\beta}{n}\leq \frac{1}{\epsilon}\lf(\frac{1}{q}
-\frac{1}{p}-\alpha\r).$$
By this and Lemmas \ref{M-JN1} and
\ref{Lem-I-HK}, we conclude that there exists a positive constant
$C$ such that, for any
$(p,q,s)_{\alpha+\beta/n}$-atom $a$,
$I_\beta (a)/C$ is a $(p,q,s,\alpha,\epsilon)$-molecule and
$$\lf\|I_\beta (a)\r\|_{HK_{(p,q,s)_\alpha}^{\mathrm{con}}(\rn)}\leq C.$$
This implies that \eqref{Ia-C} holds true.
Moreover, from (ii) of the present theorem, and Corollary \ref{Iw-JN-bound'},
we deduce that $\widetilde{I}_{\beta}$ is bounded from
$JN_{(p',q',s)_\alpha}^{\mathrm{con}}(\rn)$ to
$JN_{(p',q',s)_{\alpha+\beta/n}}^{\mathrm{con}}(\rn)$
and hence, for any $f\in JN_{(p',q',s)_\alpha}^{\mathrm{con}}(\rn)$,
$$\widetilde{I}_{\beta}(f)\in
JN_{(p',q',s)_{\alpha+\beta/n}}^{\mathrm{con}}(\rn).$$
Furthermore, by the boundedness of $\widetilde{I}_{\beta}$
from $JN_{(p',q',s)_\alpha}^{\mathrm{con}}(\rn)$ to
$JN_{(p',q',s)_{\alpha+\beta/n}}^{\mathrm{con}}(\rn)$
and Lemmas \ref{Lem-I-HK} and \ref{Ia-m},
we conclude that (i) and (ii) of Lemma \ref{Bounded-HK-A} hold true
with $A:=I_{\beta}$, $\widetilde{A}:=\widetilde{I}_{\beta}$,
$p_1:=p=:p_2$, $q_1:=q=:q_2$, $s_1:=s=:s_2$,
$\alpha_1:=\alpha+\frac{\beta}{n}$, and
$\alpha_2:=\alpha$, which further implies that
\eqref{Ia-c1} and \eqref{Ia-c2} hold true.
This finishes the proof of that (ii) $\Rightarrow$ (i)
and hence of Theorem \ref{I-Bounded-HK}.
\end{proof}

The following conclusion is a counterpart of Corollary \ref{Iw-JN-bound''}.

\begin{thm}\label{I-Bounded-HK'}
Let $\dz\in(0,1]$, $\beta\in(0,\dz)$, $p\in (1,\frac{n}{\beta})$, $q\in(1,\fz)$,
$\frac{1}{p}+\frac{1}{p'}=1=\frac{1}{q}+\frac{1}{q'}$,
$\frac{1}{\widetilde{p}}:=\frac{1}{p}-\frac{\beta}{n}$,
$s\in\zz_+$, and $\alpha\in\rr$
satisfy $\frac{1}{q}-\frac{1}{\widetilde{p}}<\alpha<\frac{s+\dz-\beta}{n}$,
$k_{\beta}$ be an $s$-order fractional kernel with regularity $\dz$,
and $I_\beta$ the fractional integral
having the vanishing moments
up to order $s$ with kernel $k_{\beta}$.
Then the following two statements are equivalent:
\begin{enumerate}
\item[\rm (i)]
there exists a positive constant $C$ such that,
for any $(p,q,s)_{\alpha}$-atom $a$,
\begin{align}\label{Ia-C'}
\lf\|I_\beta (a)\r\|_{HK_{(\widetilde{p},q,s)_\alpha}^{\mathrm{con}}(\rn)}\leq C.
\end{align}
Moreover, $I_\beta$ can be extended to a unique continuous
linear operator, still denoted by $I_\beta$,
from $HK_{(p,q,s)_{\alpha}}^{\mathrm{con}}(\rn)$
to $HK_{(\widetilde{p},q,s)_\alpha}^{\mathrm{con}}(\rn)$,
namely, there exists a positive constant
$C$ such that, for any
$g\in HK_{(p,q,s)_{\alpha}}^{\mathrm{con}}(\rn)$,
\begin{align}\label{Ia-c1'}
\lf\|I_\beta (g)\r\|_{HK_{(\widetilde{p},q,s)_{\alpha}}^{\mathrm{con}}(\rn)}
\leq C\|g\|_{HK_{(p,q,s)_{\alpha}}^{\mathrm{con}}(\rn)}
\end{align}
and, for any $f\in JN_{(\widetilde{p}',q',s)_\alpha}^{\mathrm{con}}(\rn)$,
$\widetilde{I}_{\beta}(f)
\in JN_{(p',q',s)_{\alpha}}^{\mathrm{con}}(\rn)$ and
\begin{align}\label{Ia-c2'}
\lf\langle I_\beta (g),f\r\rangle
=\lf\langle g,\widetilde{I}_\beta (f)\r\rangle,
\end{align}
where $\frac{1}{\widetilde{p}}+\frac{1}{\widetilde{p}'}=1$ and
$\widetilde{I}_{\beta}$ is as in Definition \ref{I-w} and Remark \ref{Rem-Iw}
with kernel $\widetilde{k}_{\beta}$ being the adjoint kernel of $k_{\beta}$ as in \eqref{k-w};

\item[\rm (ii)]
$I_{\beta}$ has the vanishing moments up to order $s$
as in Definition \ref{Def-I-s}.
\end{enumerate}
\end{thm}

\begin{proof}
Let $\dz$, $\beta$, $p$, $q$, $p'$, $q'$, $s$,
$\widetilde{p}$, $\widetilde{p}'$, $\alpha$, $k_\beta$,
$I_{\beta}$, $\widetilde{k}_{\beta}$,
and $\widetilde{I}_{\beta}$ be as in the present lemma.
We first prove that (i) $\Rightarrow$ (ii).
For any $f\in JN_{(\widetilde{p}',q',s)_\alpha}^{\mathrm{con}}(\rn)$,
let $\mathcal{L}_{\widetilde{I}_{\beta}(f)}$ be as in Lemma
\ref{t3.9} with $f$ replaced by $\widetilde{I}_{\beta}(f)$.
It is easy to show that $\widetilde{p}'\in(1,\frac{n}{\beta})$
and $\frac{1}{p'}=\frac{1}{\widetilde{p}'}-\frac{\beta}{n}$.
Using this, Lemma \ref{t3.9}, \eqref{Ia-c2'}, and \eqref{Ia-c1'},
we conclude that, for any
$f\in JN_{(\widetilde{p}',q',s)_\alpha}^{\mathrm{con}}(\rn)$,
\begin{align*}
\lf\|\widetilde{I}_{\beta}(f)
\r\|_{JN_{(p',q',s)_{\alpha}}^{\mathrm{con}}(\rn)}
&\sim\lf\|\mathcal{L}_{\widetilde{I}_{\beta}(f)}
\r\|_{(HK_{(p,q,s)_{\alpha}}^{\mathrm{con}}(\rn))^*}
\sim\sup_{\|g\|_{HK_{(p,q,s)_{\alpha}}^{\mathrm{con}}(\rn)}=1}
\lf|\lf\langle \mathcal{L}_{\widetilde{I}_{\beta}(f)},g\r\rangle\r|\\
&\sim\sup_{\|g\|_{HK_{(p,q,s)_{\alpha}}^{\mathrm{con}}(\rn)}=1}
\lf|\lf\langle g, \widetilde{I}_{\beta}(f)\r\rangle\r|
\sim\sup_{\|g\|_{HK_{(p,q,s)_{\alpha}}^{\mathrm{con}}(\rn)}=1}
\lf|\langle I_{\beta}(g), f\rangle\r|\\
&\ls\sup_{\|g\|_{HK_{(p,q,s)_{\alpha}}^{\mathrm{con}}(\rn)}=1}
\|I_{\beta}(g)\|_{HK_{(\widetilde{p},q,s)_\alpha}^{\mathrm{con}}(\rn)}
\|f\|_{JN_{(\widetilde{p}',q',s)_\alpha}^{\mathrm{con}}(\rn)}\\
&\ls\|f\|_{JN_{(\widetilde{p}',q',s)_\alpha}^{\mathrm{con}}(\rn)}.
\end{align*}
This shows that $\widetilde{I}_{\beta}$ is bounded from
$JN_{(\widetilde{p}',q',s)_\alpha}^{\mathrm{con}}(\rn)$ to
$JN_{(p',q',s)_{\alpha}}^{\mathrm{con}}(\rn)$,
which, together with Corollary \ref{Iw-JN-bound''},
further implies that $I_{\beta}$ has the vanishing moments up to order $s$.
This finishes the proof of that (i) $\Rightarrow$ (ii).

Now, we show that (ii) $\Rightarrow$ (i).
Let $a$ be a $(p,q,s)_{\alpha}$-atom.
Since $\frac{1}{q}-\frac{1}{\widetilde{p}}<\alpha<\frac{s+\dz-\beta}{n}$,
it follows that there exists an $\epsilon\in(0,1)$ such that
$$\frac{1}{\epsilon}
\lf(\frac{1}{q}-\frac{1}{\widetilde{p}}-\alpha\r)+\frac{1}{q'}+\frac{s}{n}<0
\quad\text{and}\quad-\frac{1}{q'}-\frac{s+\dz-\beta}{n}
\leq \frac{1}{\epsilon}\lf(\frac{1}{q}-\frac{1}{\widetilde{p}}-\alpha\r).$$
By this, the fact that $a$ is also a $(\widetilde{p},q,s)_{\alpha+\beta/n}$-atom
and Lemmas \ref{M-JN1} and \ref{Lem-I-HK},
we conclude that there exists a positive constant
$C$ such that $I_\beta (a)/C$ is a $(\widetilde{p},q,s,\alpha,\epsilon)$-molecule and
$$\lf\|I_\beta (a)\r\|_{HK_{(\widetilde{p},q,s)_\alpha}^{\mathrm{con}}(\rn)}\leq C.$$
This implies that \eqref{Ia-C'} holds true.
Moreover, using the facts $\widetilde{p}'\in(1,\frac{n}{\beta})$
and $\frac{1}{p'}=\frac{1}{\widetilde{p}'}-\frac{\beta}{n}$,
(ii) of the present theorem, and Corollary \ref{Iw-JN-bound''},
we conclude that $\widetilde{I}_{\beta}$ is bounded from
$JN_{(\widetilde{p}',q',s)_\alpha}^{\mathrm{con}}(\rn)$ to
$JN_{(p',q',s)_{\alpha}}^{\mathrm{con}}(\rn)$
and hence, for any $f\in JN_{(\widetilde{p}',q',s)_\alpha}^{\mathrm{con}}(\rn)$,
$$\widetilde{I}_{\beta}(f)\in
JN_{(p',q',s)_{\alpha}}^{\mathrm{con}}(\rn).$$
Furthermore, by the boundedness of $\widetilde{I}_{\beta}$
from $JN_{(\widetilde{p}',q',s)_\alpha}^{\mathrm{con}}(\rn)$ to
$JN_{(p',q',s)_{\alpha}}^{\mathrm{con}}(\rn)$
and Lemmas \ref{Lem-I-HK} and \ref{Ia-m},
we find that (i) and (ii) of Lemma \ref{Bounded-HK-A} hold true
with $A:=I_{\beta}$, $\widetilde{A}:=\widetilde{I}_{\beta}$,
$p_1:=p$, $p_2:=\widetilde{p}$, $q_1:=q=:q_2$, $s_1:=s=:s_2$,
$\alpha_1:=\alpha=:\alpha_2$, which further implies that
\eqref{Ia-c1'} and \eqref{Ia-c2'} hold true.
This finishes the proof of that (ii) $\Rightarrow$ (i)
and hence of Theorem \ref{I-Bounded-HK'}.
\end{proof}

\begin{rem}\label{4.25}
\begin{enumerate}
\item[\rm (i)] 	
Theorem \ref{I-Bounded-HK} implies that
the adjoint operator of the
fractional integral $I_\beta$ on
$HK_{(p,q,s)_{\alpha+\beta/n}}^{\mathrm{con}}(\rn)$
is just $\widetilde{I}_\beta$.

\item[\rm (ii)]
In Theorem \ref{I-Bounded-HK},
$\frac{1}{q}-\frac{1}{p}<\alpha<\frac{s+\delta-\beta}{n}$
is a suitable assumption
because we need $\az\in(\frac{1}{q}-\frac{1}{p},\fz)$
in Lemma \ref{M-JN1} which shows that
the molecule in Definition \ref{Def-mole} belongs to $HK_{(p,q,s)_\alpha}^{\rm{con}}(\rn)$,
and also need $\alpha\in(-\fz,\frac{s+\delta-\beta}{n})$
in Corollary \ref{Iw-JN-bound'}
which proves that $\widetilde{I}_{\beta}$ is bounded on
the congruent $\mathrm{JNC}$ space.
It is still unknown whether or not
Theorem \ref{I-Bounded-HK} holds true
with $\alpha\in (-\fz,\frac{1}{q}-\frac{1}{p}]$.

\item[\rm (iii)]
Let $p\in[1,\fz)$, $q_1=1$ and $q_2\in[1,\frac{n}{n-\beta})$,
or $q_1\in(1,\frac{n}{\beta})$ and $q_2\in[1,\frac{nq_1}{n-\beta q_1}]$,
or $q_1\in[\frac{n}{\beta},\fz)$ and $q_2\in [1,\fz)$,
$s\in\zz_+$, and $\alpha\in (-\fz,\frac{s+\dz-\beta}{n})$.
As a counterpart of Theorem \ref{Iw-JN-bound},
it is interesting to ask whether or not $I_{\beta}$
is bounded from $HK_{(p',q_2',s)_{\alpha+\beta/n}}^{\mathrm{con}}(\rn)$
to $HK_{(p',q_1',s)_{\alpha}}^{\mathrm{con}}(\rn)$,
which is a more general result than Theorem 3.31.
However, the method of this article can not give an affirmative answer
to this question because it strongly depends
on the technique of molecules and
it is still \textit{unclear} whether or not,
for any $(p',q_2',s)_\alpha$-atom $a$,
$I_{\beta}(a)$ is a $(p',q_1',s,\alpha,\epsilon)$-molecule
for some $\ez\in(0,1)$.
\end{enumerate}	
\end{rem}

\noindent \\[4mm]

\noindent\bf{\footnotesize Acknowledgements}\quad\rm
{\footnotesize
This project is partially supported by the National Natural Science Foundation of China
(Grant Nos.\
11971058, 12071197, 12122102 and 11871100) and the National
Key Research and Development Program of China
(Grant No.\ 2020YFA0712900).}\\[4mm]

\noindent{\bbb{References}}

\end{document}